\documentclass[12pt]{article}
\usepackage{amsmath,amssymb,a4wide,amsthm,mathtools}
\usepackage{hyperref}
\usepackage{enumitem}
\usepackage{xcolor}
\usepackage{comment}
\usepackage{graphicx}
\usepackage{csquotes}

\newtheorem{theorem}{Theorem}
\newtheorem{proposition}{Proposition}

\newtheorem{lemma}{Lemma}
\newtheorem{remark}{Remark}
\newtheorem{discussion}{Discussion}

\newtheorem{assumption}{Assumption}
\newtheorem{algorithm}{Algorithm}

\allowdisplaybreaks
\def\R{\mathbb{R}}
\def\N{\mathbb{N}}

\def\Lc{\mathcal{L}}
\def\Uc{\mathcal{U}}
\def\Ucs{{\mathcal{U}^*}}
\def\Vc{\mathcal{V}}

\def\Yc{\mathcal{Y}}

\def\Ft{\prescript{}{\theta}{\mathbf{F}}}

\def\Cfu{C_{f_u}}

\def\Ldf{L_{\nabla f}}

\def\jmax{\widetilde{j}^*}

\def\etildk{e_{j}^\delta}
\def\epk{\epsilon_{K(j)}}

\def\thedag{\theta^\dagger}
\def\thej{\theta^{\delta,j}}
\def\thejn{\theta^{\delta,j+1}}
\def\thekj{\widetilde{\theta}^{\delta,j}}
\def\thekjn{\widetilde{\theta}^{\delta,j+1}}
\def\thekjbar{\widetilde{\theta}^{\delta,\bar{j}}}
\def\thekjstar{\widetilde{\theta}^{\delta,\jmax(\delta)}}
\def\thekjsub{\widetilde{\theta}^{\delta_n,\jmax(\delta_n)}}
\def\uk{u_k}
\def\ukn{u_{k+1}}
\def\ztil{\widetilde{z}}
\def\Stil{\widetilde{S}_{K(j)}}
\def\ydel{y^\delta}
\def\Bthe{B_R(\theta^\dagger)}
\def\Bu{B_r(u^\dagger)}
\def\Buh{B_{r/2}(u^\dagger)}
\def\But{\prescript{}{\theta}{B}_{r/2}(u^*)}

\def\({\left(}
\def\){\right)}

\newcommand{\wto}{\rightharpoonup}

\newcommand{\embed}{\hookrightarrow}
\newcommand{\highlight}[1]{#1}

\def\R{\mathbb{R}}

\def\mv{\mathbf{m}}

\def\N{\mathbb{N}}

\def\m{\textbf{m}}

\def\h{\textbf{h}}

\def\Om2R{\Omega,\R^3}

\newcommand{\inter}{\mathop{\mathrm{int}}}

\title{Bi-level iterative regularization for inverse problems in nonlinear PDEs}
\author{Tram Thi Ngoc Nguyen\\[0.5ex] \normalsize{nguyen@mps.mpg.de}\\\normalsize{MPI Solar Systems Research - Fellow Group Inverse Problems}}
\date{ }

\begin{document}

\maketitle

\paragraph{Abstract.}We investigate the ill-posed inverse problem of recovering unknown spatially dependent parameters in nonlinear evolution PDEs. We propose a bi-level Landweber scheme, where the upper-level parameter reconstruction embeds a lower-level state approximation. This can be seen as combining the classical reduced setting and the newer all-at-once setting, allowing us to, respectively, utilize well-posedness of the parameter-to-state map, and to bypass having to solve nonlinear PDEs exactly. Using this, we derive stopping rules for lower- and upper-level iterations and convergence of the bi-level method. 
We discuss application to parameter identification for the Landau-Lifshitz-Gilbert equation in magnetic particle imaging.\\[2ex]


\section{Parameter identification}
\subsection{Introduction}\label{sec:intro}

We investigate the inverse problem of recovering the spatially dependent parameter $\theta$ in the nonlinear evolution system
\begin{equation}\label{original-pde}
\begin{split}
&\dot{u}(t)+f(t,\theta,u(t))=0\qquad t\in(0,T)\\
&u(0)=A\theta.
\end{split}
\end{equation}
with nonlinear model $f$ and linear initial condition $A$; here, $\dot{u}$ denotes time derivative of $u$. The parameter identification is carried out with additional linear measurement of the state $u$
\begin{align}\label{original-measure}
y=Lu,
\end{align}
which is usually contaminated by noise of some level $\delta$, resulting in $y^\delta$ with $\|y^\delta-y\|\leq\delta$.

The parameter identification \eqref{original-pde}-\eqref{original-measure} can be formulated in a reduced setting with the parameter-to-observation map $G$ as the forward operator
\begin{align}\label{original-red-forward}
G(\theta)=y\qquad{\text{with}}\qquad G=L\circ S:X\to\Yc,  \quad \theta\mapsto y.
\end{align}
Essentially, $S$ is a well-defined parameter-to-solution map whose output is the unique solution of the nonlinear PDE \eqref{original-pde}
\begin{align}\label{original-Smap}
S:X\to \Vc,\quad \theta\mapsto u\quad\text{solving PDE }\eqref{original-pde},
\end{align}
while $L$ is the linear observation operator 
\begin{align}\label{original-red-measure}
L:(\Vc\subseteq)\,\Uc\to\Yc, \quad u\mapsto y.
\end{align}
This abstract setting is the framework considered throughout the paper.

\paragraph{Regularization with state approximation error}Iterative regularization algorithms for solving the inverse problem \eqref{original-red-forward}-\eqref{original-red-measure}, such as Landweber-type or Newton-type methods, require the repeated evaluation of $S(\theta)$ to perform an update for the search parameter. 
In practice, one invokes some PDE solvers, for instance finite difference (FDM) or finite element methods (FEM), to instead find a numerical solution $\tilde{u}:=\widetilde{S}(\theta)\approx S(\theta)=u$. The task becomes challenging when the PDE model $f$ is nonlinear. Nonlinear solvers frequently involve an application of the fixed point theorem to handle the nonlinearity operation, which results in additional internal iterations. The convergence for numerically approximating $u$ for nonlinear PDEs is rather technical. \cite{alougesKritsikisSteinerToussaint,BanasEtAl,BartelsProhl1, cimrak} are a few examples of FEM solvers for the nonlinear Landau-Lifshitz-Gilbert equation in magnetic particle imaging (more details in Section \ref{sec:mpi}). Approximating $u$ by $\tilde{u}$ introduces additional error to the primary reconstruction process for the parameter $\theta$; as such, it is important that we consider the effect that the propagation of the \emph{state approximation error} $\epsilon:=\|\tilde{u}-u\|$, or the discretization error in case of FEM, has on our recovery of $\theta$. This is the motivating point to this work.

Inspired by this interesting line of research, in particular when the underlying PDE model is nonlinear, we investigate a general framework suitable for inverse coefficient problems. In this spirit, we consider the problem of approximating the solution to the PDE \eqref{original-pde} as another inverse problem strategically embedded into the main ill-posed parameter identification problem. 

\paragraph{Introducing the lower-level}Adapting our notation to the framework of iterative methods, we call the iteration for estimating $u$ the lower-level iteration, and the iteration for reconstructing $\theta$ the upper-level iteration. To this end, we first recast the lower level problem as an inverse problem in which the unknown is the state $u$. The forward map, at a given parameter $\theta$, is the PDE model \eqref{original-pde}
\begin{align}\label{original-pde-reform}
F(\theta):\Vc\to \Uc^*\times H \qquad
F(\theta)(u):=(\dot{u}+f(t,\theta,u(t));A\theta-u|_{t=0})
\end{align}
with the nonlinear model $f:X\times\Vc\to \Uc^*$ and the initial condition $A:X\to H$. We aim at minimizing the PDE residual towards zero, that is solving
\begin{align}\label{original-pde-reform-res}
F(\theta)(u)=\varphi:=(0,0)
\end{align}
with the image $\varphi\in\Uc^*\times H$ simply being zero. Seeking a good approximation $\widetilde{S}(\theta)$ for the exact state $S(\theta)$ in \eqref{original-Smap} is equivalent to finding an $\tilde{u}\in\Vc$ such that the PDE residual \eqref{original-pde-reform-res} is sufficiently small. Here, the choice of function space plays an important role and will be discussed in detail later.

The state approximation error $\epsilon:=\|\tilde{u}-u\|$ in the lower-level iteration must be handled with care: Allowing it to be large saves computational time, yet at the same time, it must be kept small enough such that when it enters the upper-level iteration, one can still obtain an acceptable reconstruction result for $\theta$.


\paragraph{Bi-level as a mixed approach} The suggested approach could be seen as a combined version of the reduced formulation \eqref{original-red-forward}-\eqref{original-red-measure} and an \emph{all-at-once} approach. All-at-once formulations is a recent approach introduced in optimization with PDE constraint. 
This concept was subsequently studied in the context of inverse problems \cite{BurgerMuehlhuberIP,BurgerMuehlhuberSINUM,HaAs01,aao16,KKV14b,LeHe16}, and more recently specifically for time dependent inverse problems \cite{Kaltenbacher:17,Nguyen:19}. 
In this setting, one constructs, instead of the reduced map $G$ \eqref{original-red-forward}, the higher-dimensional all-at-once map $\mathbb{G}$ and
\begin{equation}\label{aao}
\begin{split}
&\text{solve:}\quad \mathbb{G}(\theta,u)=(0,0,y)\\
&\text{with}\quad\mathbb{G}:X\times\Vc\to \Uc^*\times H\times\Yc, \quad \mathbb{G}(\theta,u):=(\dot{u}+f(t,\theta,u(t));A\theta-u|_{t=0};Lu).
\end{split}
\end{equation}
By treating $u$ as an independent unknown, one jointly recovers the parameter and the state $(\theta,u)$, bypassing the step of solving the nonlinear PDE \eqref{original-pde} exactly. This feature makes the all-at-once approach powerful in practice. A comparison between these formulations can be found in \cite{Kaltenbacher:17,Nguyen:19}. 

We remark that the classical reduced approach usually brings in certain structural conditions on the model $f$ to induce unique existence for the PDE \eqref{original-pde}. 
The new all-at-once approach \eqref{aao}, as explained, bypasses the need to study unique existence as well as solving the PDE exactly. 
This raises the question: \highlight{Can one utilize} accessible properties of $f$, and still bypass the exact solution for any nonlinear PDE? This is the viewpoint that we take in this study. In particular, we employ coercivity of the PDE model to ensure convergence of the lower-level iteration. Accordingly, as with the all-at-once approach, no nonlinear PDE needs to be solved exactly. 
Essentially, bi-level as a mixed approach combining the advantages of the reduced setting and the all-at-once setting is the new perspective of this work. 

\paragraph{Connection and contrast to other directions}
The importance of analyzing the finite-dimensional approximation error in an inverse problem framework has been pointed out in \cite{NeubauerScherzer90}; in particular, discretization error in FEM was treated in \cite{Harrach21,Clason_2016,Kaltenbacher_2011,JinZhou,Hinze_2019}.   These stimulating works specialize in inverse problems with linear PDEs, and mainly in the context of variational regularization. Here, we focus on a general framework of inverse coefficient problems with nonlinear PDEs. In addition, our research focuses on Landweber regularization instead of variational regularization. Furthermore, we do not use FEM to approximately solve the PDE, but rather take the perspective of solving the nonlinear PDE as a lower-level inverse problem.

We point to the fact that the bi-level viewpoint has been realized in optimization and is an active research area  on its own \cite{Bard}. There, the main focus is on optimality conditions rather than, as is the case in inverse problems, regularization parameters; in our case, the regularization parameters will take the form of stopping rules for the Landweber iteration. To the best of the author's knowledge, the bi-level Landweber (or gradient descent) iteration has yet to be discussed for inverse problems.

Interestingly, the idea of approximating the inverse component in a regularized iteration by another iteration can already be found in existing literature. For Newton-type method, \cite{Kaltenbacher97,Hanke97,Rieder99} propose to approximate the inverse component $[\alpha_k\text{Id}+G'(\theta_k)^*G(\theta_k)]^{-1}$ by either Landweber or Conjugate Gradient (CG) inner iteration, forming the so-called Newton-Landweber and Newton-CG methods. Independently, the inner iteration in our work is to approximate $S(\theta_k)=F(\theta_k)^{-1}$ (recalling that $G=L\circ S$), rather than terms related to the derivative $G'$. The outer iteration in our algorithm also employs the Landweber method. Following the established pattern, we might refer to our approach as a \emph{Landweber-Landweber method}.

It is worth remarking on the growing popularity of the model order reduction approach for large-scale inverse problems \cite{Frangos10}. This and our approach align in the idea of solving PDEs inexactly. If one replaces the lower-level by other inexact PDE solvers, our upper-level analysis is still well-integrated. In particular, it suggests an adaptive approximation/discretization estimation, e.g. via increasing degrees of freedom in the solvers, that can ensure convergence of the regularized reconstruction. 


\paragraph{Contribution}

In a broad sense, we exploit the structure of the inverse coefficient problem $G(\theta)=L\circ S(\theta)$, where in practice, the observation operator $L$ is linear but ill-posed, while the PDE solution map $S$ is nonlinear but well-posed. In addition, we pay attention to nonlinear time-dependent PDEs, a consideration studied relatively infrequently.

The techniques in this study are largely influenced by the seminal works \cite{scherzer95,KalNeuSch08} for the use of the \emph{tangential cone condition} in iterative regularization to handle nonlinearity. Regarding regularization parameter choice technique, we are inspired by the very recent work \cite{Kindermann17}, in which an a priori stopping scheme was introduced, as opposed to a classical posterior stopping rule. Although our proof strategy is motivated by these profound results, there are important differences between these works and our analysis. First, we investigate the Landweber method in a bi-level framework, in comparison to the standard single-level setting. Second, our work is tailored particularly to parameter identification in PDEs, rather than to the abstract formulation $F(x)=y$. Third, from a practical point of view, our bi-level algorithm (Algorithm \ref{algorithm}) is ready to use in real-life applications, as all the ingredients in the algorithm are explicitly expressed.

Concretely, we impose a coercivity assumption borrowed from PDEs theory on the lower Landweber iteration. This yields interesting outcomes such as verification of tangential cone condition (Lemma \ref{coerive-alpha=1}), convergence rate (Theorem \ref{low-conv-rate}) and monotonicity of the residual sequence (Proposition \ref{low-monotone-res}). For the upper Landweber iteration, we carefully control the interaction between the state approximation error and the noise level, and their propagation through the adjoint process and iterative update. As our main result, we provide explicit adaptive stopping rules, both for lower and upper iteration, confirming convergence of the bi-level scheme (Theorems \ref{theo:bi-priorstop}).

\paragraph{Structure} The paper can be outlined as follows: After an introduction of the problem setting in Section \ref{sec:setting}, we present the bi-level algorithm in Section \ref{sec:bi-level}. Section \ref{sec:lowiter} is dedicated to studying the lower-level problem. Section \ref{sec:upper} advances the investigation to the upper-level problem. Section \ref{sec:discus} verifies the proposed assumptions and discusses the application to magnetic particle imaging. Finally, we summarize our findings and future prospects in Section \ref{sec:conclude}.

\subsection{Problem setting}\label{sec:setting}

\subsubsection*{Setting}\label{dis-setting}
\begin{itemize}[label=,leftmargin=*]
\item The space of the spatially dependent parameter is $X$, and the data space  is denoted by $\Yc$. Notionally, calligraphic letters indicate spaces that involve the time variable $t$.
\item The state space is $\Uc:=\{u\in L^2(0,T;U)\}$, with $U\embed H\embed U^*$ forming a Gelfand triple. 
\item The \emph{smooth} state space is $\Vc:=\{u\in \Uc:\dot{u}\in \Uc^*\}=\{u\in L^2(0,T;U): \dot{u}\in L^2(0,T;U^*)\}$. 
\item The image space, which the PDE residual belongs to, is $\Uc^*=L^2(0,T;U^*)$. Our choice of $\Vc$ ensures well-definedness and boundedness of $\frac{d}{dt}:\Vc\to\Uc^*$. Unlike the reduced formulation, where $\Uc^*$ does not appear, in the all-at-once setting \eqref{aao} and the combined setting  proposed in \eqref{original-pde-reform}, the PDE image space $\Uc^*$ plays a clear role, guaranteeing well-definedness for  the lower-level problem. 

The state space $\Vc$ must be sufficiently smooth for well-definedness of the PDE \eqref{original-pde} and existence of an adjoint state (Lemma \ref{lem:adjoint}). This function space setting is inspired by the book \cite[Chapter 8]{Roubicek}.
\item
$X$, $\Yc$, $\Uc$, $\Vc$, $\Uc^*$ are all Hilbert spaces, ensuring that our problem remains in a Hilbert space framework.  

\item The parameter-to-state (or solution) map $S: X\ni\theta\mapsto u\in\Vc(\subseteq\Uc)$ is nonlinear and assumed to be Fr\'echet differentiable.
\item The observation (or measurement) operator $L:\Uc\ni u\mapsto y\in\Yc$ is linear and bounded. In real life, $L$ is usually a compact operator; it cannot capture smoothness of $u\in\Vc$ due to discrete or noisy measurements. Its domain $\Uc$ hence does not need to be as smooth as $\Vc$. 
\item  The choice of data (or observation) space $\Yc$ depends on how data is acquired, e.g. full measurements, partial measurements or \highlight{snapshots} of $u$; thus we keep $\Yc$ as a general Hilbert space. The supercsript $\delta$ in $y^\delta$ denotes the noise level.
\item  The initial value operator $A:X\to H$ is linear and bounded. The continuous embedding $\Vc\embed C(0,T;H)$ \cite[Section 7.2]{Roubicek} enables the time-wise evaluation of $u$, ensuring well-definedness of the setting $u(t_0)=A\theta\in H$.
\item The PDE model $f:(0,T)\times X\times U\to U^*$ is nonlinear and is assumed to satisfy the Carath\'eodory condition. This means that for each $\theta\in X$, the mapping $f$ is measurable with respect to $t$ for any $u\in U$, and continuous with respect to $u$ for almost every ﬁxed $t\in(0,T)$. In this way, $f$ can be naturally supposed to induce the similarly denoted Nemytskii operator 
\[f:X\times\Vc\to \Uc^* \qquad [f(\theta,u)](t):=f(t,\theta,u(t)).\]
We refer to \cite{Roubicek, Troeltzsch} for details on Carath\'eodory and Nemytskii mappings.
\item Finally, $\Omega$ is a smooth and bounded spatial domain. The time line $[0,T]$ is finite.
\end{itemize}

\subsubsection*{Notation}
\begin{itemize}[label=,leftmargin=*]
\item
We use the capitalized subscript $R$ for constants relating to the upper-level problem (e.g. \eqref{cond-upper-boundedG'}-\ref{cond-upper-tcc}), while the lower-case subscript $r$ is reserved for constants in the lower-level problem (e.g. \eqref{low-ass-derivative}-\eqref{low-ass-tcc}). The subscript $S$ is used in constants relating to the parameter-to-state map $S$ (e.g. \eqref{cond-S'}). Regarding iteration index, the superscript $(\cdot)^j$ denotes the iterates of the upper-level problem, and subscript $(\cdot)_k$ is used for the lower-level problem. 
\item 
$B^*$ denotes the Hilbert space adjoint, and $B^\star$ denotes the Banach space adjoint of a linear, bounded operator $B$. The notation $(\cdot,\cdot)$ means the inner product, and $\langle\cdot,\cdot\rangle$ is the dual paring (c.f. Proposition \ref{lem:adjoint}).
\item
$D_U:U\to U^*$ and $I_X:X^*\to X$ are isomorphisims (c.f. Proposition \ref{lem:adjoint}). As an example, consider $U=X=H^1_0(\Omega)$. In this case, $D_U$ might be a differential operator, e.g. $(\text{Id}-\Delta)$; and $I_X$ could be an integral operator, e.g. $(\text{Id}-\Delta)^{-1}$.
\item 
$\|B\|_{X\to Y}$ denotes the operator norm of $B:X\to Y$. In most of the cases, where it is clear which norms are used due to the definition of the operators, we abbreviate, e.g. $\|L\|:=\|L\|_{\Uc\to\Yc}, \|A\|:=\|A\|_{X\to H}$ etc. Nevertheless, in some computations when technicality is involved, we explicitly write out the norm. We also write various other norms in shortened form, e.g. $\|\cdot\|_{L^2}:=\|\cdot\|_{L^2(\Omega)}$, $\|u\|_{L^2(L^2)}:=\|u\|_{L^2(0,T;L^2(\Omega))}$.
\item The constant $C_{X\to Y}$ indicates the norm of the continuous embedding $X\embed Y$.
\item 
Throughout the paper, we assume that the PDE \eqref{original-pde}-\eqref{original-measure} is solvable, and denote by $\theta^\dagger$ the exact parameter and $u^\dagger$ the corresponding exact solution. $\Bthe\subset X$ is the ball of radius $R$ around the ground truth $\theta^\dagger$; similar interpretation applies for the ball $\Bu\subseteq\Vc$. 
\item
For the lower-level problem in Section \ref{sec:lowiter}, $\But\subseteq\Vc$ denotes the ball of radius $r/2$ centered at $u^*$, where $u^*$ is a solution of $F(\theta)(u)=\varphi$ at a given $\theta$ (see \eqref{low-IP}).

$f'_\theta, f'_u$ are the partial derivatives of $f$, while $\nabla f$ is its gradient.
\end{itemize}

\section{Bi-level iteration}\label{sec:bi-level}
\subsection{Bi-level algorithm}\label{sec:algorithm}
We now establish the bi-level Landweber algorithm \ref{algorithm}. The main idea is to first perform upper-level iteration for $\theta$ by the standard Landweber iteration with the exact forward map $G=L\circ S$. We then replace the parameter-to-state map $S$ by an approximation $\Stil$ obtained through a lower-level Landweber iteration. $\Stil$ enters the forward and backward (adjoint) evaluation in the upper-level iteration, yielding a gradient update step for $\theta$. After Algorithm \ref{algorithm} is constructed, we perform some preliminary error analysis in Section \ref{sec:pre-erroranalysis}.\\

First, recalling the set up of the forward map $G(\theta)=L\circ S(\theta)$ in \eqref{original-red-forward}-\eqref{original-red-measure}, the standard Landweber iteration for $\theta$ in this case reads as
\begin{align*}
\thejn = \thej -S'(\thej)^*L^*\big(LS(\thej)-\ydel\big) \qquad j\leq j^*(\delta).
\end{align*}
Replacing $S$ by its approximation $\Stil$, the Landweber iteration with $\Stil$ becomes
\begin{align}
\thekjn = \thekj -S'(\thekj)^*L^*\(L\Stil(\thekj)-\ydel\) \qquad j\leq\jmax(\delta),
\end{align}
where at each $j$, that is, each $\thekj$, the state approximation $\Stil(\thekj):=u^j_{K(j)}$ is the outcome of the lower-level iteration
\begin{align}
u^j_{k+1}=u^j_k-F_u'(\thekj,u^j_k)^*\(F(\thekj,u^j_k)-\varphi\) \qquad k< K(j)
\end{align}
with the PDE map $F(\thekj)$ described in \eqref{original-pde-reform}-\eqref{original-pde-reform-res}.
The stopping rule $K(j)$ for the lower-level iteration as well as $\jmax$ for the upper-level iteration will be the subject of the later sections. 
\\

In order to see how the state approximation enters the adjoint $S'(\thekj)^*$, we first of all derive the adjoint process. 
\begin{proposition}[Adjoint]\label{lem:adjoint}
Let $D_U:U\to U^*$ and $I_X:X^*\to X$ be isomorphisms.\\
Then, the Hilbert space adjoint of $G'$ with $G$ defined in \eqref{original-red-forward}-\eqref{original-red-measure} is given by
\begin{align}
&G'(\theta)^*:\Yc\to X && G'(\theta)^*= S'(\theta)^*L^* \label{adjoint}\\[1ex]
&L^*:\Yc \to \Uc \nonumber\\ 
&S'(\theta)^*:\Uc\to X && S'(\theta)^*v = -\int_0^T I_Xf'_\theta(\theta,u)(t)^\star z(t)\,dt +I_XA^\star z(0), \label{adjoint-S'}
\end{align}
where $z$ is the solution to the final value problem at source $D_Uv$
\begin{align}\label{low-ass-adjointeq}
\begin{cases}
&-\dot{z}(t)+f'_u(\theta,u)(t)^\star z(t)= D_Uv(t) \qquad  t\in(0,T)\\
&z(T)=0.
\end{cases}
\end{align}
Here, the Banach space adjoints are
\begin{align*}
&f'_u(t,\theta,u(t)):U\to U^*,\qquad  &&f'_\theta(t,\theta,u(t)):X\to U^*, \qquad && A:X\to H=L^2(\Omega), \\
&f'_u(t,\theta,u(t))^\star:U^{**}=U\to U^*,&& f'_\theta(t,\theta,u(t))^\star:U\to X^*, && A^\star:H^*=H\to X^*. 
\end{align*}
\end{proposition}

\begin{proof}
As the parameter-to-state map $S:X\to\Vc$ in \eqref{original-Smap} is assumed to differentiable, $S'(\theta)\xi=:p\in\Vc$ solves the sensitivity (or linearized) equation
\begin{equation}
\begin{split}\label{low-sensitive-eq}
&\dot{p}(t)+f'_u(\theta,u)(t)p(t)= -f'_\theta(\theta,u)(t)\xi \qquad  t\in(0,T)\\
&p(0)=A\xi,
\end{split}
\end{equation}
where $u=S(\theta)$ is the solution to the original PDE \eqref{original-pde}.
Note that $p\in\Vc$, with $\Vc$ being smooth as defined in Section \ref{dis-setting}, along with the Nemytskii assumption on $f$ ensure $\dot{p}+f'_u(\theta,u)p\in \Ucs$ and $f'_\theta(\theta,u)\xi\in \Ucs$, confirming well-definedness of the sensitivity equation. 

In the following computation, we will respectively use \eqref{low-ass-adjointeq}, integration by parts and \eqref{low-sensitive-eq} with the helps of isomorphisms to translate inner products to dual pairings. For every $\xi\in X, v\in\Uc$, one has
\begin{align*}
&\(S'(\theta)\xi,v\)_\Uc = \int_0^T\( p, v\)_Udt=\int_0^T\langle p, D_Uv\rangle_{U,U^*}dt\\
&= \int_0^T\int_\Omega \langle p,-\dot{z}+f'_u(\theta,u)^\star z \rangle_{U,U^*}\,dt\,dx\\
& = \int_0^T\int_\Omega \(\dot{p}+f'_u(\theta,u)p \)z\,dt\,dx + \int_\Omega A\xi\, z(0)\,dx
=\int_0^T \langle \dot{p}+f'_u(\theta,u)p,z\rangle_{U^*,U} dt + \(A\xi, z(0)\)_H\\
&=\int_0^T \langle-f'_\theta(\theta,u)\xi, z\rangle_{U^*,U}\,dt + \langle \xi, A^\star z(0)\rangle_{X,X^*}\\
&=\int_0^T \langle\xi, -f'_\theta(\theta,u)^\star z\rangle_{X,X^*}\,dt +  \langle \xi, A^\star z(0)\rangle_{X,X^*}\\
&=\(\xi, I_X\int_0^T -f'_\theta(\theta,u)^\star z + I_XA^\star z(0)\)_X
=:\(\xi, S'(\theta)^*v\)_X.
\end{align*}
This proves that the adjoints of $S'(\theta)$ are as shown in \eqref{adjoint-S'}. Applying the chain rule, we obtain the adjoint of the composition map
\[G'(\theta)^*=(L\circ S'(\theta))^*=S'(\theta)^*L^*\]
as claimed in \eqref{adjoint}.
\end{proof}

With the adjoints detailed, we can now state the Landweber iterations for the exact problem (single-level) and the approximated problem (bi-level).\\

\hrule
\begin{algorithm}\,\label{algorithm}
\begin{enumerate}
\item \textbf{\{Single-level\} Iterations with exact $S$.}\\[1ex]
Update parameter:
\begin{equation}
\begin{split}
&\theta^{\delta,0}\in\Bthe, \\
&\thejn = \thej -S'(\thej)^*L^*\(LS(\thej)-\ydel\)\qquad j\leq j^*(\delta),
\end{split}
\end{equation}
 where at each $j$: 
\begin{align}
\text{solve exactly PDE }\eqref{original-pde} \text{ at $\thej$ for } u=:S(\thej)
\end{align}
then compute adjoint
\begin{equation}\label{al:adjoint}
\begin{split}
&v:=L^*\(LS(\thej)-\ydel\)\\
&S'(\thej)^*v = -\int_0^T I_Xf'_\theta(\thej,S(\thej))(t)^\star z(t)\,dt +I_XA^\star z(0) \\[0.5ex]
&\begin{cases}
&-\dot{z}(t)+f'_u(\thej,S(\thej))^\star z(t)= D_Uv(t) \qquad  t\in(0,T)\\
&z(T)=0.
\end{cases}
\end{split}
\end{equation}
\item \textbf{\{Bi-level\} Iterations with approximated $\Stil$.}\\[1ex]
\{Upper-level\} Update parameter:
\begin{equation}
\begin{split}\label{al:approx-theta}
&\widetilde{\theta}^{\delta,0}\in\Bthe\\
& \thekjn = \thekj -S'(\thekj)^*L^*\(L\Stil(\thekj)-\ydel\) \qquad j\leq \jmax(\delta),
\end{split}
\end{equation}
\{Lower-level\} where at each $j$, meaning at each $\thekj$, iterate :
\begin{align}\label{al:approx-S}
& u^j_0\in\But  \nonumber\\
& u^j_{k+1}=u^j_k-F_u'(\thekj,u^j_k)^*\(F(\thekj,u^j_k)-\varphi\) \qquad k\leq K(j)\\
& u^j_{K(j)} = :\Stil(\thekj). \nonumber
\end{align}
\{back to Upper-level\}Then compute the approximate adjoint 
\begin{equation}\label{al:approx-adjoint}
\begin{split}
&v:=L^*\(L\Stil(\thej)-\ydel\)\\
&S'(\thekj)^*v = -\int_0^T I_Xf'_\theta(\thekj,\Stil(\thekj) )(t)^\star\ztil(t)\,dt +I_XA^\star z(0) \\[0.5ex]
&\begin{cases}
&-\dot{\ztil}(t)+f'_u(\thekj,\Stil(\thekj) )^\star \ztil(t)= D_Uv(t) \qquad  t\in(0,T)\\
&\ztil(T)=0.
\end{cases}
\end{split}
\end{equation}
\end{enumerate}
\end{algorithm}
\hrule
\begin{remark}\label{rem:algorithm}
In the bi-level algorithm, no exact $S$ appears; equivalently, no nonlinear PDE needs to be exactly solved. The state approximation $\Stil(\thekj)$ enters the residual $\(L\Stil(\thekj)-\ydel\)$ as well as the adjoint \eqref{al:approx-adjoint}. We emphasize that the term $S'(\thekj)^*$ appearing here is not, in fact, the adjoint of the derivative of the exact $S$ evaluated at $\thekj$, but an approximation thereof.

\end{remark}

\subsection{Preliminary error analysis}\label{sec:pre-erroranalysis}
At this stage, the Landweber algorithms for single-level and bi-level have been established. The two algorithms differ in the state and the adjoint as mentioned in Remark \ref{rem:algorithm}. We thus examine the \emph{system output error} and the \emph{adjoint error} with respect to a given approximation error $\epk$; this is the focus of this section. 

The state approximation error $\epk$ will be quantified in the succeeding section. The subscript $K(j)$ alludes to the fact that we will eventually derive a stopping rule $K$ of the lower-level with respect to the upper-level iterate $j$. For now, we assume that $K(j)$ is given.

\begin{lemma}[System output error]\label{lem-outputerror}

Denote the state approximation error $\epk$ at the input $\thekj\in B_R(\theta^\dagger)$ by
\begin{align}
\epk:=\|S(\thekj)-\Stil(\thekj)\|_\Uc,
\end{align}
and assume that there exists a constant $M_S>0$ such that
\begin{align}\label{cond-S'}
\|S'(\theta)\|_{X\to\Uc}\leq M_S \qquad \forall \theta\in B_R(\theta^\dagger).
\end{align}
Given any other input $\theta^j\in B_R(\theta^\dagger)$, the system output error is 
\begin{align}\label{error-system}
\|LS(\theta^j)-L\Stil(\thekj)\|\leq M_S\|L\|\|\thej-\thekj\| + \|L\|\epk.
\end{align}
\end{lemma}
\begin{proof}
Differentiability of $S:X \to \Vc$, thus of $S:X \to \Uc$ due to $\Vc\embed\Uc$, and application of the mean value theorem imply
\begin{align*}
S(y)-S(x)=\int_0^1S'(x+\lambda(y-x))\,d\lambda(y-x)
\end{align*}
meaning
\begin{align*}
\|LS(\theta^j)-L\Stil(\thekj)\|_\Yc&\leq \|LS(\theta^j)-LS(\thekj)\|_\Yc+\|LS(\thekj)-L\Stil(\thekj)\|_\Yc\\
&\leq\|L\int_0^1 S'(\thekj-\lambda(\theta^j-\thekj))\,d\lambda(\theta^j-\thekj) \|_\Yc + \|L\|_{\Uc\to\Yc}\epk\\
&\leq\|L\|_{\Uc\to\Yc}\sup_{\theta\in \Bthe}\|S'(\theta)\|_{X\to\Uc}\|\theta^j-\thekj \|_X + \|L\|_{\Uc\to\Yc}\epk\\
&\leq M_S\|L\|\|\theta^j-\thekj\| + \|L\|\epk.
\end{align*}
\end{proof}

\begin{remark}\label{rem:bounddev}
$M_S$ in \eqref{cond-S'} is a locally uniform bound for solutions of the linearized PDE \eqref{low-sensitive-eq}. 
A sufficient condition for \eqref{cond-S'} to hold is Lipschitz continuity of the parameter-to-state map $S$. When $\theta$ plays the role of a source term or initial data, results regarding Lipschitz continuity of $S$ for general classes of nonlinear PDE model $f$ exist in the literature \cite[Theorem 8.32]{Roubicek}. 

\end{remark}

We now turn our attention to the discrepancy between the exact adjoint and the approximate adjoint with respect to the state approximation error $\epk$.

\begin{lemma}[Adjoint state error]\label{lem-adjointerror}
Let the assumptions in Lemma \ref{lem-outputerror} hold. 
Suppose that the solution $z$ of the adjoint PDE
\begin{equation*}
\begin{cases}
&-\dot{z}+f'_u(\theta,u)^\star z= h \\
&z(T)=0,
\end{cases}
\end{equation*}
depends continuously on the source term
\begin{align}\label{cond-linearizedPDE}
\|z\|_\Uc+\|z(0)\|_H\leq \Cfu \|h\|_\Ucs \qquad \forall (\theta,u)\in \Bthe\times\Bu.
\end{align}
Furthermore, assume that $f'_\theta,f'_u$ are Lipschitz continuous, with
\begin{equation}
\begin{split}\label{cond-lipf'}
&\|f'_\theta(\theta_1,u_1)-f'_\theta(\theta_2,u_2)\|_{X\to\Ucs}+\|f'_u(\theta_1,u_1)-f'_u(\theta_2,u_2)\|_{\Uc\to\Ucs}\\
&\hspace{7cm}\leq \Ldf \(\|\theta_1-\theta_2\|_X+\|u_1-u_2\|_\Uc\),
\end{split}
\end{equation}
and denote $
L_{R,r}:=\Ldf\(R+rC_{\Vc\to\Uc}\)+\|f'_\theta(\theta^\dagger,u^\dagger)\|_{X\to\Ucs}+\|f'_u(\theta^\dagger,u^\dagger)\|_{\Uc\to\Ucs}$.\\[1ex]
Then, given any other input $\theta^j\in B_R(\theta^\dagger)$, we obtain the adjoint error
\begin{equation}
\begin{split}\label{error-adjoint}
\|S'&(\theta^j)^*-S'(\thekj)^*\|_{\Uc\to X}\\
&\leq  \Ldf\Cfu\|I_X\|\|D_U\|\Big[L_{R,r} \Cfu+1+\|A\|\Cfu\Big]\((1+M_S)\|\theta^j-\thekj\|+ \epk\)\\
&=:C_{\nabla f,A}\((1+M_S)\|\theta^j-\thekj\|+ \epk\).
\end{split}
\end{equation}
\end{lemma}
\begin{proof}
The Lipschitz condition \eqref{cond-lipf'} implies boundedness of the derivatives, due to
\begin{align}\label{bound-f'}
&\|f'_\theta(\theta,u)\|_{X\to\Ucs}+\|f'_u(\theta,u)\|_{\Uc\to\Ucs}\nonumber\\&\leq\|f'_\theta(\theta,u)-f'_\theta(\theta^\dagger,u^\dagger)\|+\|f'_u(\theta,u)-f'_u(\theta^\dagger,u^\dagger)\| +\|f'_\theta(\theta^\dagger,u^\dagger)\|+\|f'_u(\theta^\dagger,u^\dagger)\| \nonumber\\
&\leq \Ldf\big(\|\theta-\theta^\dagger\|_X+C_{\Vc\to\Uc}\|u-u^\dagger\|_\Vc\big)+\|f'_\theta(\theta^\dagger,u^\dagger)\|+\|f'_u(\theta^\dagger,u^\dagger)\| \nonumber\\
&\leq \Ldf\big(R+rC_{\Vc\to\Uc}\big)+\|f'_\theta(\theta^\dagger,u^\dagger)\|+\|f'_u(\theta^\dagger,u^\dagger)\| \nonumber\\
&=L_{R,r},
\end{align}
recalling that $R$ and $r$ are the radii of the balls $\Bthe\subset X$ and $\Bu\subset \Vc$. We will use this bound in evaluating the term $Q_2$ later in the proof.

Consulting Algorithm \ref{algorithm}, subtracting the approximate adjoint state \eqref{al:approx-adjoint} from the exact adjoint state \eqref{al:adjoint}, then inserting the mixed term $f'_\theta(\thekj,\Stil(\thekj))^*z$, we see that for any $v\in \Uc$,
\begin{align*}
&\(S'(\theta^j)-S'(\thekj)\)^*v\\&=-I_X\int_0^T \(f'_\theta(\theta^j,S(\theta^j))(t)^\star z(t)- f'_\theta(\thekj,\Stil(\thekj) )(t)^\star \ztil(t)\)dt+I_XA^*(z(0)-\ztil(0))\\
&=-I_X\int_0^T \(f'_\theta(\theta^j,S(\theta^j))(t)^\star -f'_\theta(\thekj,\Stil(\thekj) )(t)^\star\)z(t)\,dt\\&\qquad\quad- I_X\int_0^T f'_\theta(\thekj,\Stil(\thekj))(t)^\star(z(t)-\ztil(t))\,dt\\&\quad\qquad+I_XA^\star(z(0)-\ztil(0))\\[1ex]
&=:Q_1+Q_2+Q_3.
\end{align*}
We recall that $z$ and $\ztil$ are, respectively, the solutions of
\begin{equation*}
\begin{cases}
&-\dot{z}+f'_u(\theta^j,S(\theta^j))^\star z= D_Uv \\
&z(T)=0,
\end{cases}
\text{ and }
\begin{cases}
&-\dot{\ztil}+f'_u(\thekj,\Stil(\thekj) )^\star \ztil= D_Uv\\
&\ztil(T)=0.
\end{cases}
\end{equation*}
Inspecting $Q_2$, we see that the difference $z-\ztil$ solves
\begin{equation}
\begin{cases}
&-(\dot{z}-\dot{\ztil})+f'_u(\thekj,\Stil(\thekj) )^\star (z-\ztil)= \(-f'_u(\theta^j,S(\theta^j)+f'_u(\thekj,\Stil(\thekj)\)^\star z\\
&\ztil(T)=0.
\end{cases}
\end{equation}
Now, using \eqref{cond-linearizedPDE}, we can bound $z$ by the source term $D_Uv$. In a similar fashion, we bound $z-\ztil$ by the source term $(-f'_u(\theta^j,S(\theta^j)+f'_u(\thekj,\Stil(\thekj))^\star z$, then estimate this source further via the Lipschitz condition \eqref{cond-lipf'}. More precisely, we have
\begin{align*}
&\|z\|_\Uc\leq \Cfu\|D_Uv\|_\Ucs\leq \Cfu\|D_U\|_{U\to U^*}\|v\|_\Uc,\\[2ex]
&\|z-\ztil\|_\Uc+\|z(0)-\ztil(0)\|_H\\
&\qquad\leq \Cfu \|f'_u(\theta^j,S(\theta^j)-f'_u(\thekj,\Stil(\thekj)\|_{\Uc\to\Ucs}\|z\|_{\Uc}\\
&\qquad\leq \Ldf (\Cfu)^2\|D_U\|_{U\to U^*}\(\|\theta^j-\thekj\|_X+\|S(\theta^j)-\Stil(\thekj)\|_\Uc\)\|v\|_\Uc.
\end{align*}
Using these and the boundedness of the derivative from \eqref{bound-f'}, we can estimate $Q_1$, $Q_2$ and $Q_3$ by the parameter difference $\|\theta^j-\thekj\|$ and their difference propagating through the state approximation $\|S(\theta^j)-\Stil(\thekj)\|$, i.e.
\begin{align*}
\|Q_2\|_X&\leq \|I_X\|_{X^*\to X}\|f'_\theta(\thekj,\Stil(\thekj))^\star\|_{\Uc\to X^*}\|z-\ztil\|_\Uc\\
&\leq \|I_X\|L_{R,r}\Ldf (\Cfu)^2\|D_U\| \(\|\theta^j-\thekj\|_X+\|S(\theta^j)-\Stil(\thekj)\|_\Uc\) \|v\|_\Uc,\\[1ex]
\|Q_1\|_X&\leq \|I_X\|\Ldf \(\|\theta^j-\thekj\|_X+\|S(\theta^j)-\Stil(\thekj)\|_\Uc\) \Cfu\|D_U\|_{U\to U^*}\|v\|_\Uc\\
&=\|I_X\|\Ldf\Cfu\|D_U\| \(\|\theta^j-\thekj\|_X+\|S(\theta^j)-\Stil(\thekj)\|_\Uc\)\|v\|_\Uc,\\[1ex]
\|Q_3\|_X&\leq \|I_X\|_{X^*\to X}\|A^\star\|_{H\to X^*}\|z(0)-\ztil(0)\|_H\\
&\leq \|I_X\|\|A\| \Ldf (\Cfu)^2\|D_U\|\(\|\theta^j-\thekj\|_X+\|S(\theta^j)-\Stil(\thekj)\|_\Uc\)\|v\|_\Uc.
\end{align*}
In $Q_2$, we have applied the estimate $\|f'_\theta(\thekj,\Stil(\thekj))^\star\|_{\Uc\to X^*}\leq L_{R,r}$ derived in \eqref{bound-f'}.

In the last step, combining $Q_1$, $Q_2$ and $Q_3$, then applying Lemma \ref{lem-outputerror} for $L=\text{Id}$ to further estimate $\|S(\theta^j)-\Stil(\thekj)\|_\Uc$, we arrive at
\begin{align*}
&\|(S'(\theta^j)-S'(\thekj))^*\|_{\Uc\to X}\\
&\leq\|I_X\| \Ldf\Cfu\|D_U\|\Big[L_{R,r} \Cfu+1+\|A\|\Cfu\Big]\(\|\theta^j-\thekj\|_X+\|S(\theta^j)-\Stil(\thekj)\|_\Uc\)\\
&\leq \|I_X\| \Ldf\Cfu\|D_U\|\Big[L_{R,r} \Cfu+1+\|A\|\Cfu\Big]\((1+M_S)\|\theta^j-\thekj\|+ \epk\),
\end{align*}
which is \eqref{error-adjoint}; this completes the proof.
\end{proof}

At this point, we have analyzed the system output error and adjoint error in the bi-level Algorithm \ref{algorithm} via the state approximation error $\epk$. We now study $\epk$ itself.

\section{The lower-level iteration}\label{sec:lowiter}
In this section, we approximate $S$ by $\widetilde{S}$ at given $\theta\in \Bthe$ via the procedure \eqref{al:approx-S} in Algorithm \ref{algorithm}. In particular, we examine convergence and convergence rate of \eqref{al:approx-S} in order to derive the state approximation error $\epk:=\|S(\thekj)-\Stil(\thekj)\|_\Uc$ in an explicit form. Optimally, $\epk$ should be controllably small by letting the lower-level scheme iterate until a sufficiently large stopping index $K(j)$ is reached. The lower-level stopping index $K(j)$ will be derived in the next section, when its role in the convergence of the upper-level iteration becomes relevant.

For convenience, we shorten the notation in the lower-level problem \eqref{original-pde-reform}. Explicitly, for a fixed $\theta$, we introduce the model operator
\begin{align}\label{low-IP}
\Ft:\Vc\to\Uc^*\times H\qquad \Ft(u):=F(\theta)(u)=(\dot{u}+f(\theta,u);A\theta-u|_{t=0})=\varphi;
\end{align}
thus, $\Ft'$ will indicate the derivative with respect to the state $u$, and an exact solution $u^*$ to \eqref{low-IP} for a given $\theta$ satisfies \[\Ft(u^*)=\varphi.\] 
The procedure \eqref{al:approx-S} in Algorithm \ref{algorithm}, dropping the superscript $j$, now takes the form
\begin{equation}\label{low-iter}
\begin{split}
& u_0\in\But \\
& u_{k+1}=u_k-\Ft'(u_k)^*\(\Ft(u_k)-\varphi\) \qquad k\leq K(j)
\end{split}
\end{equation}
Finally, we recall the use of $(\theta^\dagger,u^\dagger)$ to denote the exact parameter and exact corresponding state of the original inverse problem \eqref{original-pde}-\eqref{original-measure}. 

The following result presents the convergence analysis for the lower-level Landweber iterates. 
We derive the convergence rate in the noise-free case not via the well-known source condition, but rather through a coercivity assumption. 
The weak convergence analysis is carried out 
without assuming weak closedness of $\Ft$. 

\begin{theorem}[Convergence and rate] \label{low-conv-rate}
Suppose that at each $\theta\in\Bthe$, the equation $\Ft(u)=\varphi$ is solvable at some $u^*\in\Buh$.

i) Assume that there exists a constant $M_r>0$ such that 
\begin{align}\label{low-ass-derivative}
&\|\Ft'(u)\|_{\Vc\to\Ucs\times H}\leq M_r<\sqrt{2} \qquad \forall u\in\Bu
\end{align}
and a constant $\mu_r>0 $ such that the weak tangential cone condition (wTC)
\begin{equation}\label{low-ass-tcc}
\begin{split}
2\(\Ft(u)-\Ft(v),\Ft'(u)(u-v)\)_{\Ucs\times H}\geq (M_r^2+\mu_r)\|\Ft(u)-\Ft(v)\|_{\Ucs\times H}^2\\  \forall u,v\in\Bu
\end{split}
\end{equation}
holds uniformly for all $\theta\in\Bthe$.\\
Then, the whole sequence $\{u_k\}$ as given in \eqref{al:approx-S} and \eqref{low-iter} remains in the ball $\Bu$ for all $k\in\N$. Furthermore, the sequence converges weakly to a solution of $\Ft(u)=\varphi$.\\

ii) If additionally $\Ft$ is coercive, in the sense that there exists some $C_{coe}>0$ such that
\begin{align}\label{low-ass-coercive}
C_{coe}\|\Ft(u)-\Ft(u^*)\|_{\Ucs\times H}\geq \|u-u^*\|_\Vc^\alpha\qquad \text{for some }\alpha\geq1
\end{align}
for all $u\in\Bu$ and all $\theta\in\Bthe$, then we obtain strong convergence of $\{u_k\}$, with rate
\begin{align}\label{low-rate}
\|u_k-u^*\|_\Vc=\mathcal{O}\(\frac{1}{k^{1/(2\alpha)}}\), \quad \|\Ft(u_k)-\varphi\|_{\Ucs\times H}=\mathcal{O}\(\frac{1}{k^{1/(2\alpha)}}\).
\end{align}
\end{theorem}

Before proceeding with the proof, we make some observation on these assumptions.
\begin{remark}[On assumptions of Theorem \ref{low-conv-rate}]\,\label{rem-derivative}
\begin{enumerate}[label=(\roman*)]
\item 
A sufficient condition for \eqref{low-ass-derivative} is that $\|f'_u(\theta,u)\|_{X\times\Vc\to\Uc^*}$ is uniformly bounded in $\Bthe\times\Bu$. This is due to the fact that we already have boundedness of the linear operators $\|\frac{d}{dt}\|_{\Vc\to\Uc^*}, \|A\|_{X\to\Uc^*}$ and $\|(\cdot)_{t=0}\|_{\Vc\to\Uc^*}$ guaranteed from the choice of function spaces suggested in Section \ref{dis-setting}. The magnitude condition $M_r<\sqrt{2}$ is easily attained by scaling.

\item
Assumption \eqref{low-ass-tcc} is equivalent to
\[\(\Ft(u)-\Ft(v)-\Ft'(u)(u-v),\Ft(u)-\Ft(v)\)\leq \(1-\frac{M_r^2+\mu_r}{2}\)\|\Ft(u)-\Ft(v)\|^2.\]
This is the form of the weak tangential cone condition introduced in \cite{scherzer95,Kindermann21} with the tangential cone constant $c_{tc}:=\(1-\frac{M_r^2+\mu_r}{2}\)\in(0,1)$. Its strong version (sTC) takes the form
\[\|\Ft(u)-\Ft(v)-\Ft'(u)(u-v)\|\leq \tilde{c}_{tc}\|\Ft(u)-\Ft(v)\| \qquad \tilde{c}_{tc}\in(0,1).\]
Moreover, as the initial condition operator $A$ is linear, (wTC) and (sTC) in fact constrain only the nonlinear model $f$ here; moreover, this constraint is with respect to the state $u$ only, and not to the parameter $\theta$ (see also Remark \ref{dis:TCC}). 

\item
Assumption \eqref{low-ass-coercive} can be relaxed to the weaker norm $\|u-u^*\|^\alpha_\Uc$. In this case, \eqref{low-rate} yields the rate $\|u_k-u^*\|_\Uc=\mathcal{O}\(\frac{1}{k^{1/(2\alpha)}}\)$; this rate is sufficient for the treatment of the upper-level problem in the later sections, as we need only quantify $\epk:=\|S(\thekj)-\Stil(\thekj)\|_\Uc$, the state approxmation in $\Uc$-norm.
\end{enumerate}
\end{remark}

\begin{proof}[Proof of Theorem \ref{low-conv-rate}]
i) Fix $\theta\in\Bthe$ and $\Ft(u^*)=\varphi$. Following the classical approach \cite{HNS95,KalNeuSch08}, we show Fej\'er monotonicity of the error sequence via the boundedness assumption \eqref{low-ass-derivative} on the derivative and the weak tangential cone condition \eqref{low-ass-tcc}:
\begin{equation}\label{low-Fejer}
\begin{split}
\|\ukn-u^*\|^2&-\|\uk-u^*\|^2= 2\(\ukn-\uk,\uk-u^*\)+\|\ukn-\uk\|^2 \\
&=2\(\varphi-\Ft(\uk),\Ft'(\uk)(\uk-u^*)\)+\|\Ft'(\uk)^*(\Ft(\uk)-\varphi)\|^2 \\
&\leq -(M_r^2+\mu_r)\|\Ft(\uk)-\varphi\|^2+M_r^2\|\Ft(\uk)-\varphi\|^2 \\
&= -\mu_r \|\Ft(\uk)-\varphi\|^2. 
\end{split}
\end{equation} 
This estimate shows that for each $\theta$, as long as $u_0$ is inside the ball $\But$, the whole sequence $u_k$ remains in $\But$. Together with the fact $u^*\in\Buh$, it follows that $u_k\in\Bu$ for all $\theta\in\Bthe$. As a consequence, there exists subsequence $\{u_{k_n}\}$ weakly convergent to some $\bar{u}\in\Vc$, and \[u_{k_n}\wto\bar{u}\in \Bu\] as $\Bu$ is closed and convex.

Next, summing the inequality \eqref{low-Fejer} from 0 up to any arbitrary index $K\in\N$  results in
\begin{align}\label{low-sum-F}
\mu_r\sum_{k=0}^K \|\Ft(\uk)-\varphi\|^2&\leq \|u_0-u^*\|^2-\|u_{K+1}-u^*\|^2\leq \|u_0-u^*\|^2 \qquad \forall K\in\N \nonumber\\
\Rightarrow\sum_{k=0}^\infty \|\Ft(\uk)-\varphi\|^2& \leq \frac{1}{\mu_r}\|u_0-u^*\|^2.
\end{align}
Thus, together with the uniform bound \eqref{low-ass-derivative} on the derivative, we have convergence of the image and the gradient step series;
\begin{align*} 
&\Ft(\uk)\to\varphi\qquad\text{as }k\to\infty\text{ and}\\
&\sum_{k=0}^\infty\|u_{k+1}-u_k\|^2=\sum_{k=0}^\infty\|\Ft'(\uk)^*(\Ft(\uk)-\varphi)\|^2\leq \frac{M_R^2}{\mu_R} \|u_0-u^*\|^2.
\end{align*}

Now as $\bar{u}\in\Bu$, applying again the weak tangential cone condition \eqref{low-ass-tcc} then inserting $\varphi$ and using the weak convergence result $u_{k_n}\wto\bar{u}$ and the strong convergence $\Ft(\uk)\to\varphi$, we have
\begin{align*}
&\lim_{n\to\infty}\frac{M_r^2+\mu_r}{2}\|\Ft(\bar{u})-\Ft(u_{k_n})\|^2\leq \lim_{n\to\infty}\(\Ft(\bar{u})-\Ft(u_{k_n}),\Ft'(\bar{u})(\bar{u}-u_{k_n})\)\\
&=\lim_{n\to\infty}\(\Ft'(\bar{u})^*(\Ft(\bar{u})-\varphi-\Ft(u_{k_n})+\varphi),\bar{u}-u_{k_n}\)\\
&=\lim_{n\to\infty}\(\Ft'(\bar{u})^*(\Ft(\bar{u})-\varphi),\bar{u}-u_{k_n}\)+\lim_{n\to\infty}\(\Ft'(\bar{u})^*(-\Ft(u_{k_n})+\varphi),\bar{u}-u_{k_n}\)\\
&=\lim_{n\to\infty} \(\Ft'(\bar{u})^*(-\Ft(u_{k_n})+\varphi),\bar{u}-u_{k_n}\)\leq rM_r\lim_{n\to\infty}\|\Ft(u_{k_n})-\varphi\|\\&=0.
\end{align*}
This implies \[\Ft(\bar{u})=\varphi,\] proving that the weak limit $\bar{u}$ of any weakly convergent subsequence is in fact an exact solution. 

Next, we deduce weak convergence of the whole sequence $\{u_k\}$ from Fej\'er monotonicity. Similar to \cite[Lemma 9.1]{ClasonValkonen}, let $u^*$, $\bar{u}$ be two weak accumulation points, thus exact solutions, as argued above. Then there exist subsequences $\{u_{k_n}\}$, $\{u_{k_m}\}$ with $u_{k_n}\wto u^*$ and $u_{k_m}\wto\bar{u}$ when $n$, $m\to\infty$. Accordingly,
\begin{align*}
\|u^*-\bar{u}\|^2&=\lim_{n,m\to\infty}(u_{k_n},u^*-\bar{u})-(u_{k_m},u^*-\bar{u})\\
&=\lim_{n,m\to\infty}\frac{1}{2}\(\|u_{k_n}-u^*\|^2-\|u^*\|^2-\|u_{k_n}-\bar{u}\|^2+\|\bar{u}\|^2\)\\&\hspace{1.5cm} -\frac{1}{2}\(\|u_{k_m}-u^*\|^2-\|u^*\|^2-\|u_{k_m}-\bar{u}\|^2+\|\bar{u}\|^2\)\\
&=\lim_{n,m\to\infty} \frac{1}{2}\(\|u_{k_n}-u^*\|^2-\|u_{k_n}-\bar{u}\|^2\)-\frac{1}{2}\(\|u_{k_m}-u^*\|^2-\|u_{k_m}-\bar{u}\|^2\)\\&=\frac{1}{2}(a-b)-\frac{1}{2}(a-b) =0.
\end{align*} 
In the last line, these four limits exist since $\{u_{k_n}\},\{u_{k_m}\}$ are extracted from the iterate sequence $\{u_k\}$, which has monotone decreasing distance to any true solution, as shown in \eqref{low-Fejer}. As these limits coincide, using a sub-subsequence argument, we conclude weak convergence of the full sequence to a solution, that is,
\begin{align*}
u_k\overset{k\to\infty}\wto u^*\quad\text{with}\quad \Ft(u^*)=\varphi.
\end{align*}

ii) For the convergence rate, with coercivity \eqref{low-ass-coercive} (thus $u^*$ is unique) and strong convergence of the image, we have strong convergence of $\uk$ to the exact $u^*$; in particular,
\begin{align*}
\|\uk-u^*\|^\alpha\leq C_{coe}\|\Ft(\uk)-\Ft(u^*)\|=C_{coe}\|\Ft(\uk)-\varphi\|\to0 \quad\text{as }k\to\infty.
\end{align*}
By Fej\'er monotonicity, we obtain the rate 
\begin{align}\label{low-ballx}
\|\uk-u^*\|^{2\alpha} < \frac{1}{k}\sum_{j=0}^k\|u_j-u^*\|^{2\alpha} \leq \frac{C_{coe}^2}{k} \sum_{j=0}^k\|\Ft(u_j)-\varphi\|^2\leq  \frac{C_{coe}^2/\mu_r}{k} \|u_0-u^*\|^2=\mathcal{O}\(\frac{1}{k}\)
\end{align} 
for the iterate sequence; for the image sequence, we similarly have
\begin{align}\label{low-bally}
\|\Ft(u_k)-\varphi\|=\left\|\int_0^1\Ft'(u^*-\lambda(\uk-u^*))\,d\lambda(\uk-u^*)\right\|\leq M_r \|\uk-u^*\| =\mathcal{O}\(\frac{1}{k^{(1/2\alpha)}}\).
\end{align}
\end{proof}

\begin{remark}[Coercivity - PDE theory]\label{rem-coerivity} Coercivity \eqref{low-ass-coercive} is one of the key elements to analyzing unique existence for PDEs \cite{Evans, Roubicek} . This condition is often stated for the nonlinear model $f$; after some derivation, it may then be transformed into semi-coercivty of the PDE model, which is $\Ft$ in our case.  

This, as alluded to in Section \ref{sec:intro}, is the motivating point for our study. We have employed well-posedness of the parameter-to-state map $S$, which is induced from coercivity of the model $\Ft$, to iteratively approximate the state $u$ with convergence guarantee and rate. 

At the end of this paper (Section \ref{sec:verify-coercive}, Proposition \ref{prop-PDEcoercive}, Remark \ref{rem-PDEcoercive-V}), we will estblish coercivity for some classes of nonlinear parabolic PDEs. This demonstrates that \eqref{low-ass-coercive} is indeed  a suitable and verifiable assumption for the lower-level problem. 

\end{remark}

\begin{remark}[Coercivity - conditional stability]\label{rem-statbility}
Assumption \eqref{low-ass-coercive} means H\"older continuity of the inverse of $\Ft$ and can be extended to superadditive index functions $\phi$, that is,
\begin{equation}
\begin{split}\label{stability}
&\phi:[0,\infty)\to[0, \infty)\text{ superadditive, monotonically increasing, } \lim_{t\to 0}\phi(t)=0\\
&\|u-v\|^2\leq \phi\big(\|\Ft(u)-\Ft(v)\|^2\big) \quad \forall u, v\in\text{a given compact set } B.
\end{split}
\end{equation}
 Indeed, \eqref{low-ballx} then becomes
\begin{align*}
\|\uk-u^*\|^2&< \frac{1}{k}\sum_{j=0}^k\|u_j-u^*\|^{2} \leq \frac{1}{k} \sum_{j=0}^k\phi\left(\|\Ft(\uk)-\varphi\|^2\right)\leq \frac{1}{k} \phi\left(\sum_{j=0}^k\|\Ft(\uk)-\varphi\|^2\right)\\&\leq  \frac{1}{k} \phi\(\frac{1}{\mu_r}\|u_0-u^*\|^2\)=\mathcal{O}\(\frac{1}{k}\)
\end{align*} 
implying the linear rate. The extended version \eqref{stability} can be seen as the conditional stability property of a continuous and injective operator. 

There is also a relation to the variational source condition
\[\ell(u,u^*)\leq \mathcal{R}(u)-\mathcal{R}(u^*)+\phi\(\|\Ft(u)-\Ft(u^*)\|^2\),\] in which the link from the loss functional $\ell$ to the data penalty via the index function $\phi$ is the key to convergence rates in variational regularization \cite{Flemming}.
\end{remark}

Going back to our setting where $\Ft$ describes a PDE, 
we focus on coercivity \eqref{low-ass-coercive} rather than on the extension \eqref{stability}. In fact, coercivity leads us to the following result.

\begin{lemma}[Strong tangential cone condition (sTC)]\label{coerive-alpha=1}
The strong tangential cone condition (see Remark \ref{rem-derivative}-ii)) holds under coercivity \eqref{low-ass-coercive} and the H\"older continuity assumption
\begin{align}\label{low-ass-Lips}
\|\Ft'(u)-\Ft'(v)\|_{\Vc\to\Uc^*\times H}\leq L_{F'}\|u-v\|_\Vc^\beta, \qquad \text{with }
\beta> \alpha-1\geq0
\end{align}
on the derivative, with some $L_{F'}\geq0$, for all $u,v\in\Bu$, provided the ball $\Bu$ satisfies
\begin{align}\label{r-small}
r<\(\frac{1+\beta}{L_{F'}C_{coe}}\)^{1/(1+\beta-\alpha)}.
\end{align}
In this scenario, all the claims in Theorem \ref{low-conv-rate} remain valid.
\end{lemma}
\begin{proof}
\eqref{low-ass-Lips} and coercivity as in \eqref{low-ass-coercive} yield the strong tangential cone condition
\begin{align*}
\|&\Ft(u)-\Ft(u^*)-\Ft'(u)(u-u^*)\|_{\Uc^*\times H}=\left\|\int_0^1\Ft'(u^*+\lambda(u-u^*))-\Ft'(u)\,d\lambda(u-u^*) \right\|\\
&\leq\left\|\int_0^1\Ft'(u^*+\lambda(u-u^*))-\Ft'(u)\,d\lambda\right\|_{\Vc\to \Uc^*\times H}\|u-u^*\|_\Vc \\
&\leq \frac{L_{F'}}{1+\beta}\|u-u^*\|_\Vc^{1+\beta} \\
&\leq \frac{L_{F'}C_{coe}}{1+\beta}\|u-u^*\|_\Vc^{1+\beta-\alpha}\|\Ft(u)-\Ft(u^*)\|_{\Uc^*\times H}\leq \frac{L_{F'}C_{coe}}{1+\beta}r^{1+\beta-\alpha}\|\Ft(u)-\Ft(u^*)\|_{\Uc^*\times H}\\
&=:\tilde{c}_{tc}\|\Ft(u)-\Ft(u^*)\|_{\Uc^*\times H},
\end{align*}
for $u,u^*\in\Bu$. Above, the strict inequality \eqref{r-small}
ensures that the tangential cone constant $\tilde{c}_{tc}$ is less than one.
\end{proof}
\begin{remark}[Convergence is highly feasible]\label{rem:conv-lower}
Lemma \ref{coerive-alpha=1} shows that with coercivity, one only needs H\"older continuity of the derivative as in \eqref{coerive-alpha=1} to confirm convergence of the lower-level iteration. Indeed, the special case of $\beta=1$, that is, that the derivative is Lipschitz continuous, is a natural assumption for convergence of gradient descent in many contexts, e.g. in convex optimization.
\end{remark}

Continuing along these lines, we observe some connections between the weak tangential cone condition and other renowned conditions yielding convergence of the gradient descent method, such as convexity and the Polyak-Lojasiewicz (PL) condition. For consider the cost functional
$ J(u):=\|\Ft(u)-\varphi\|^2 $.
Recall that $J$ is said to satisfy the PL condition \cite{KarimiNutiniSchmidt} if there exists a $\mu>0$ such that for all $u$,
\begin{align*}
\|j'\|_{\Vc^*}^2\geq\mu\(J(u)-\min_u J(u)\) \qquad j'\in\partial J(u).
\end{align*}

\begin{lemma}[Convexity versus wTC versus Polyak-Lojasiewicz condition]\label{cor:TCC-PL}
Assume that the equation $\Ft(u)=\varphi$ is solvable at $u^*\in\Bu$. Let $\Ft$ be differentiable with (rescaling as necessary) $\|\Ft'(u)\|\leq M_r<1, \forall u\in\Bu$.
We have the relation:
\begin{align*}
J\text{ is } convex\quad\Longleftrightarrow\quad &\text{wTC with } M_r^2+\mu_r=1\\ &\qquad\qquad\big\Downarrow\\&\text{wTC standard } \eqref{low-ass-tcc} \quad\overset{\text{ coercivity }\eqref{low-ass-coercive} \text{ with }\alpha=1} \Longrightarrow\quad  J \text{ is PL}
\end{align*}
\end{lemma}
\begin{proof}
i) As $J$ is Fr\'echet differentiable with
\[J'(u)=2\Ft'(u)^*(\Ft(u)-\varphi),\]
the subdifferential $\partial J(u)$ is a singleton. Furthermore, convexity yields
\begin{align*}
j'\in\partial J(u)&\Leftrightarrow j'(u-u^*)\geq J(u)-J(u^*)\\
&\Leftrightarrow 2\(\Ft'(u)^*(\Ft(u)-\varphi),u-u^*\) \geq \|\Ft(u)-\varphi\|^2-\|\Ft(u^*)-\varphi\|^2\\
&\Leftrightarrow 2\(\Ft(u)-\Ft(u^*),\Ft'(u)(u-u^*)\) \geq \|\Ft(u)-\Ft(u^*)\|^2
\end{align*}
for $\Ft(u^*)=\varphi$ and any $u\in\Vc$. Restricting our consideration to $u\in\Bu$, this is in fact the wTC \eqref{low-ass-tcc} with the fixed constant $(M_r^2+\mu_r)=1$; clearly, this implies the standard wTC.

ii) As $\Ft(u^*)=\varphi$, we have $\min_u J(u)=0$. Thus, if the wTC \eqref{low-ass-tcc} holds and coercivity \eqref{low-ass-coercive} is attained at $\alpha=1$, then
\begin{align*}
\|J'(u)\|_{\Vc^*}&=\sup_{\|v\|_\Vc=1}2\(\Ft'(u)^*(\Ft(u)-\varphi),v\)\geq\frac{2}{\|u-u^*\|}\(\Ft(u)-\Ft(u^*),\Ft'(u)(u-u^*)\)  \\
&\geq \frac{M_r^2+\mu_r}{\|u-u^*\|}\|\Ft(u)-\Ft(u^*)\|^2\geq \frac{M_r^2+\mu_r}{C_{coe}}\|\Ft(u)-\varphi\|\\
&=:\sqrt{\mu J(u)},
\end{align*}
showing that the PL condition is fulfilled.
\end{proof}

Going back to the central point of this section, monotonicity of the error sequence is the main character in Theorem \ref{low-conv-rate}; however, the same question for the image sequence in the nonlinear case, has, to the best of the author's knowledge, yet to be discussed. The following proposition gives sufficient conditions for monotonicity of the residual. 

\begin{proposition}[Monotonic residual]\label{low-monotone-res}
Let the conditions in Theorem \ref{low-conv-rate} be fulfilled, and assume moreover that \eqref{low-ass-coercive} holds with $\alpha=1$ for any\footnote{That is, not only at a solution $v=u^*$.} $u,v\in\Bu$
Assume that $\Ft'(u_k)^*:\Vc\to\Uc^*\times H$ is injective, and moreover
\begin{align}\label{low-ass-mu}
C_{coe}^2\(C_{coe}^2M_r^2-1\)<\mu_r.
\end{align}
Then, the residual sequence $\|\Ft(\uk)-\varphi\|$ decreases monotonically.
\end{proposition}
\begin{proof}
Consider $u\in \inter\Bu$, the interior of $\Bu$, and $v:=u+\epsilon h\in \inter \Bu$ with sufficiently small $\epsilon>0$. Employing the coercivity condition \eqref{low-ass-coercive}, we have
\begin{align*}
\|\Ft'(u)h\|_{\Uc^*\times H}=\lim_{\epsilon\to 0}\frac{\|\Ft(u)-\Ft(u+\epsilon h)\|_{\Uc^*\times H}}{\epsilon}\geq\frac{1}{C_{coe}}\|h\|_{\Vc} \qquad \forall h\in\Vc.
\end{align*}
The open map theorem implies that at any $u\in \inter \Bu$, the bounded linear form $\Ft'(u)\in\Lc(\Vc,\Uc^*\times H)$ has closed range. Equivalently, the  adjoint $\Ft'(u)^*$ also has closed range. This, together with injectivity, gives us that $\Ft'(u_k)^*$ is also bounded below by the same constant
\begin{align}\label{closed-range}
\|\Ft'(u)^*y\|_{\Vc}\geq \frac{1}{C_{coe}}\|y\|_{\Uc^*\times H} \qquad \forall y\in \Uc^*\times H.
\end{align}
Note that this estimate holds uniformly for all $\Ft$ with $\theta\in\Bthe$.

Next, using the weak tangential cone condition \eqref{low-ass-tcc} and the parallelogram law, one gets
\begin{align*}
&(M_r^2+\mu_r)\|\Ft(\ukn)-\Ft(\uk)\|^2\leq 2\(\Ft(\uk)-\Ft(\ukn),\Ft'(\uk)(\uk-\ukn)\)\\
&=2\big(\Ft'(\uk)^*(\Ft(\uk)-\Ft(\ukn)),\Ft'(\uk)^*(\Ft(\uk)-\varphi)\big)\\
&=\|\Ft'(\uk)^*(\Ft(\uk)-\varphi)\|^2 - \|\Ft'(\uk)^*(\Ft(\ukn)-\varphi)\|^2 + \|\Ft'(\uk)^*(\Ft(\uk)-\Ft(\ukn))\|^2.
\end{align*}
Applying boundedness of the derivative as in \eqref{low-ass-derivative} to the first and to the third term of the right hand side, then moving the third term to the left hand side, we have
\begin{align*}
\|\Ft'(\uk)^*(\Ft(\ukn)-\varphi)\|^2+\mu_r\|\Ft(\ukn)-\Ft(\uk)\|^2 \leq M_r^2\|(\Ft(\uk)-\varphi)\|^2.
\end{align*}
Recall that as long as the initial guess satisfies $u_0\in B_{r/2}(u^*)$, all the iterates $u_k$ remain in the ball. Moreover, $u_k$ belongs to the interior of $B_{r/2}(u^*) $ due to Fe\'jer monotonicity \eqref{low-Fejer}; thus, $u_k\in \inter  B_{r/2}(u^*)\subset \inter \Bu$. This enables us to apply the lower bound \eqref{closed-range} to the first term in the above estimate, with $u=u_k$, $y=\Ft(\ukn)-\varphi$, thus obtaining
\begin{align}\label{est-residual}
\frac{1}{C_{coe}^2}\|\Ft(\ukn)-\varphi\|^2+\mu_r\|\Ft(\ukn)-\Ft(\uk)\|^2 \leq M_r^2\|\Ft(\uk)-\varphi)\|^2.
\end{align}
Here, we notice that the lower bound $1/C_{coe}$ cannot exceed $M_r$, the upper bound for $\|\Ft'(u)\|$; hence, we cannot claim decay of the image sequence yet. Therefore, we further estimate the middle term of \eqref{est-residual}, again invokinging coercivity \eqref{low-ass-coercive} and the lower bound \eqref{closed-range}
\begin{align*}
\mu_r\|\Ft(\ukn)-\Ft(\uk)\|^2&\geq \frac{\mu_r}{C_{coe}^2}\|\uk-\ukn\|^2=\frac{\mu_r}{C_{coe}^2}\|\Ft'(\uk)^*(\Ft(\uk)-\varphi)\|^2\\&\geq \frac{\mu_r}{C_{coe}^4}\|\Ft(\uk)-\varphi\|^2.
\end{align*}
Substituting this into \eqref{est-residual}, we arrive at
\begin{align*}
\|\Ft(\ukn)-\varphi\|^2\leq C_{coe}^2\(M_r^2-\frac{\mu_r}{C_{coe}^4}\)\|\Ft(\uk)-\varphi)\|^2< \|\Ft(\uk)-\varphi)\|^2,
\end{align*} 
where the strict inequality is achieved according to the constraint \eqref{low-ass-mu}. This shows the claimed monotonicity of the residual sequence.

\end{proof}

\begin{remark}[Boundedness from below of $\Ft'(u_k)$ and its adjoint]
Boundedness from below of $\Ft'(u_k)$ and its adjoint $\Ft'(u_k)^*$ is in fact the same as invertibility of $\Ft'(u_k)$, thus of $\Ft'(u_k)^*$.
This ties to unique existence analysis of the linearized PDE \eqref{low-sensitive-eq} and the adjoint equation \eqref{low-ass-adjointeq}, which appears in the iterative Algorithm \ref{algorithm}.
\end{remark}

We conclude this section by discussing some links from the lower-level to the upper-level, stepping toward parameter recovery.

\begin{discussion}\,
\begin{itemize}
\item
Convergence rate of the state approximation error $\epk:=\|u_k-u^*\|_\Uc$ in the $\Uc$-norm is sufficient to carry out the analysis of the upper-level (c.f. Remark \ref{rem-derivative}-iii)). In some practical examples, this simplifies verification of the coercivity assumption \eqref{low-ass-coercive}.
\item
Consider the lower-level problem $\Ft(u)=\varphi$, $\theta\in\Bthe$. As long as the exact state $u^*$ and the initialization $u_0$, which might depend on $\theta$, lie in $\Buh$, then all iterates will remain in the ball $\Bu$ (Theorem \ref{low-conv-rate}), which is independent of $\theta$. This indicates that if we ensure that in the upper-level problem, the parameter iterates $\theta_j$ do not escape the ball $\Bthe$, then the lower-level iterates $u_k$ are guaranteed to remain in the ball $\Bu$.
\item
The lower-level problem is noise-free, as $\varphi$ is the residual of the PDE model. We can therefore iterate $u_k$ until the error $\epk$ is acceptably small, then input this approximate state into the upper-level iteration. This gives rise to a lower-level stopping rule $k\leq K$, where the stopping index $K$ depends on $j$-th iterate of the upper-level iteration and the noise level $\delta$. We will determine $K(j,\delta)$ in the next section. 

\item
The state approximation error $\epk$ in practical experiments can be observed via plotting the PDE residual. Indeed, they behave in a similar manner (Theorem \ref{low-conv-rate}, ii) rate), and the PDE residual even decays monotonically (Proposition \ref{low-monotone-res}), echoing Fej\'er monotonicity of the error (Theorem \ref{low-conv-rate}, \eqref{low-Fejer}).
\end{itemize}
\end{discussion}
\medskip

We summarize the assumptions that have been used.
\begin{assumption}[Bi-level algorithm]\label{summary-ass-bi}
For all $\theta$, $\hat{\theta}\in\Bthe$ and for all $u$, $\hat{u}\in\Bu$, the following all hold:
\begin{enumerate}
\item $S'$ is uniformly bounded, with
\[
\|S'(\theta)\|_{X\to\Uc}\leq M_S
\]
\item The solution $z$ to the adjoint PDE depends continuously on the source $h$, that is,
\begin{align*}
\begin{cases}
&-\dot{z}+f'_u(\theta,u)^\star z= h \\
&z(T)=0
\end{cases} \quad\text{implies}\quad \|z\|_\Uc+\|z(0)\|_H\leq \Cfu \|h\|_\Ucs.
\end{align*}
\item The derivatives $f'_\theta$, $f'_u$ satisfy the Lipschitz condition
\begin{align*}
&\|f'_\theta(\theta,u)-f'_\theta(\hat{\theta},\hat{u})\|_{X\to\Ucs}+\|f'_u(\theta,u)-f'_u(\hat{\theta},\hat{u})\|_{\Uc\to\Ucs}\leq \Ldf \(\|\theta-\hat{\theta}\|_X+\|u-\hat{u}\|_\Uc\)
\end{align*}
\end{enumerate}
for some positive constants $M_S,\Cfu,\Ldf$.
\end{assumption}
\smallskip

\begin{assumption}[Lower-level iteration]\label{summary-ass-low}
$\Ft(u):=F(\theta,u)=\varphi$ is solvable at $u^*$. Moreover, for all $\theta\in\Bthe$ and for all $u$, $\hat{u}\in\Bu$, the following all hold:
\begin{enumerate}\setcounter{enumi}{3}
\item $\Ft'$ is bounded, with
\[\|\Ft'(u)\|_{\Vc\to\Ucs\times H}\leq M_r<\sqrt{2},\]
\item Either the weak tangential cone condition
\[2\(\Ft(u)-\Ft(\hat{u}),\Ft'(u)(u-\hat{u})\)_{\Ucs\times H}\geq (M_r^2+\mu_r)\|\Ft(u)-\Ft(\hat{u})\|_{\Ucs\times H}^2,\]
holds, or $\Ft'$ satisfies the H\"older condition (see Lemma \ref{coerive-alpha=1})
\begin{align}
\|\Ft'(u)-\Ft'(v)\|_{\Vc\to\Uc^*\times H}\leq L_{F'}\|u-v\|_\Vc^\beta,\qquad \beta>0.
\end{align}
\item $\Ft$ is coercive, with
\[C_{coe}\|\Ft(u)-\Ft(u^*)\|_{\Ucs\times H}\geq \|u-u^*\|_\Vc^\alpha\quad (\text{or } \geq \|u-u^*\|_\Uc^\alpha),\qquad\alpha\geq1\]
\end{enumerate}
with positive constants $L_{F'}$, $C_{coe}$, $M_r$, $\mu_r$.
\end{assumption}

\section{The upper-level iteration}\label{sec:upper}

In the preceeding, we approximated the PDE solution by the lower-level iteration. The approximate state $u^j_{K(j)} = :\Stil(\thekj)$ and the associated approximate adjoint will be now fed into the upper-level of Algorithm \ref{algorithm} to update the parameter $\thekj$. This section studies convergence of the upper-level iteration under the influence of the noise level $\delta$ of the measurement, state approximation and adjoint approximation errors. In particular, we discuss two types of stopping rule: a posterior choice inspired by the classical result \cite{KalNeuSch08}, and a prior choice inspired by the newer result \cite{Kindermann17}. We put more emphasis on the latter.

Let us recall from \eqref{original-red-forward} that the reduced map $G=L\circ S$ is a composition of the linear, bounded observation $L$ and the nonlinear, differentiable parameter-to-state-map $S$. The forward map $G$ has uniformly bounded derivative, satisfying
\begin{equation}\label{cond-upper-boundedG'}
\begin{split}
\|G'(\theta)\|_{X\to Y}=\|L\circ S'(\theta)\|_{X\to Y}\leq \|S'(\theta)\|_{X\to\Uc}\|L\|_{\Uc\to Y}\leq M_S\|L\|=:M_R
\end{split}
\end{equation}
for all $\theta\in\Bthe$, with $M_S$ as in Assumption \ref{summary-ass-bi}.
\begin{assumption}[Upper-level iteration]\label{summary-ass-up}
\begin{enumerate}\,\setcounter{enumi}{6}
\item Strong tangential cone condition (sTC): for all $\theta$, $\hat{\theta}\in\Bthe$, the relations 
\begin{equation}\label{cond-upper-tcc}
\begin{split}
2\(G(\theta)-G(\hat{\theta}),G'(\theta)(\theta-\hat{\theta})\)_\Yc&\geq(M_R^2+\mu_R)\|G(\theta)-G(\hat{\theta})\|_\Yc^2,\\[1ex]
\|G'(\theta)(\theta-\hat{\theta})\|_\Yc&\leq K_R\|G(\theta)-G(\hat{\theta})\|_\Yc
\end{split}
\end{equation}
\end{enumerate}
hold for some positive constants $M_R$, $\mu_R$, $K_R$ with $(M_R^2+\mu_R)<2$.
\end{assumption}
The first part of \eqref{cond-upper-tcc} is the weak the tangential cone condition, while the second part strengthens the assumption (c.f. Discussion \ref{rem-derivative}, ii)). Assumption \eqref{cond-upper-tcc} can also be written as
\begin{equation}\label{cond-upper-tcc1}
\begin{split}
\|G(\theta)-G(\hat{\theta})-G'(\theta)(\theta-\hat{\theta})\|_\Yc\leq C_{tc}\|G&(\theta)-G(\hat{\theta})\|_\Yc
\end{split}
\end{equation}
for all $\theta$, $\hat{\theta}\in\Bthe$, with $C_{tc}:=1-(M_R^2+\mu_R)/2\in(0,1)$ and $K_R=1+C_{tc}$. Unlike the tangential cone condition w.r.t $u$ of the lower-level problem, the derivative here is taken w.r.t the parameter $\theta$.

Algorithm \ref{algorithm} presents the upper-level iteration 
\begin{align}\label{al-upper}
&\widetilde{\theta}^{\delta,0}\in\Bthe  \nonumber\\
& \thekjn = \thekj -S'(\thekj)^*L^*\(L\Stil(\thekj)-\ydel\) \qquad j\leq \jmax(\delta),
\end{align}
where $\Stil(\thekj)$ is approximated according to the lower-level iteration \eqref{al:approx-S} and $S'(\thekj)^*$ is computed according to the adjoint process \eqref{al:approx-adjoint}.

\subsection{A posterior stopping rule}\label{sec:upper-posterior}
In this section, we analyze 
an a posteriori choice for the stopping index $\jmax$ of the upper-level iteration. The analysis also yields a stopping rule $K(j)$ for the lower-level problem. 


\begin{proposition}[Boundedness - Fej\'er monotnicity]\label{prop-upper-Fejer} 
Let Assumptions \ref{summary-ass-bi}, \ref{summary-ass-low}, \ref{summary-ass-up} hold, and let $\thekj\in\Bthe$ be given. Then $\thekjn$ in \eqref{al-upper} is a better approximation, i.e. \[\|\thekjn-\thedag\|<\|\thekj-\thedag\|\] 
(in particular, the whole sequence remains in the ball $\Bthe$) if, for some $\gamma(j)\geq 0$, the two following conditions are satisfied:

\begin{enumerate}
\item The lower-level \eqref{al:approx-S} is iterated until step $K(j)$, such that the state approximation error satisfies
\begin{align}\label{cond-lower-stopindex}
\mathcal{O}\(\frac{1}{K(j)^{1/(2\alpha)}}\) = \epk\leq\gamma(j)\|LS(\thekj)-y^\delta\|
\end{align}
with some nonnegative and sufficiently small $\gamma(j)$.

\item In the upper-level, the residual satisfies the relation
\begin{align}\label{upper-posterior-choice}
& \|LS(\thekj)-y^\delta\|>\delta \,\Gamma(j) > 0, \\
\Gamma(j) & := 2\dfrac{1+C_{tc}+\|L\|K_R{\gamma(j)}}{2-(1+\epsilon)M_R^2-2C_{tc}- \|L\|\big( 2K_R{\gamma(j)}+M_R^2\|L\|{\gamma(j)^2} \big)\(1+\frac{1}{\epsilon}\)} \nonumber
\end{align}
for some $\epsilon>0$ 
with $C_{tc}$ as in \eqref{cond-upper-tcc1} and $M_R$, $\mu_R$, $K_R$ as in \eqref{cond-upper-tcc}.

\end{enumerate}

\end{proposition}

\begin{proof}
Starting from \eqref{al-upper}, by inserting $LS(\thekjn)$ and applying Young's inequality $(a+b)^2\leq (1+\epsilon) a^2+(1+1/\epsilon)b^2$, $\epsilon>0$, one estimates the error discrepancy between two consecutive iterations as
\begin{align}\label{auxi-esstimate}
&\etildk:=\|\thekjn-\thedag\|^2-\|\thekj-\thedag\|^2\nonumber\\
&=2\( \thekjn-\thekj,\thekj-\thedag \)+\|\thekjn-\thekj\|^2 \nonumber\\
&=-2\( L\Stil(\thekj)-y^\delta,LS'(\thekj)(\thekj-\thedag) \) \nonumber\\
& \qquad +\|S'(\thekj)^*L^*(L\Stil(\thekj)-y^\delta)\|^2 \nonumber\\
&\leq-2\( LS(\thekj)-y^\delta,LS'(\thekj)(\thekj-\thedag) \)-2\( L\Stil(\thekj)-LS(\thekj),LS'(\thekj)(\thekj-\thedag) \) \nonumber\\
&+(1+\epsilon)\|S'(\thekj)^*L^*(LS(\thekj)-y^\delta)\|^2+\(1+\frac{1}{\epsilon}\)\|S'(\thekj)^*L^*(L\Stil(\thekj)-LS(\thekj))\|^2 \nonumber\\
&=: E_1+E_2+E_3+E_4, 
\end{align}
recalling that $S(\theta)$ is the exact evaluation, while $\Stil(\theta)$ is the approximation via the lower-level iteration \eqref{al:approx-S}.

For $E_1$, we insert $y^\delta-LS(\thekj)$. 
For $E_2$, we recall the state approximation error $\epk:=\|S(\thekj)-\Stil(\thekj)\|$, and use the second part of the tangential cone condition \eqref{cond-upper-tcc}. For $E_3$ and $E_4$, we employ boundedness of the derivative from \eqref{cond-upper-boundedG'}, as well as the definition of $\epk$ for $E_4$. Explicitly, one has 
\begin{align*}
E_1&= -2\|LS(\thekj)-y^\delta\|^2-2\(LS(\thekj)-y^\delta, y^\delta-LS(\thekj)+LS'(\thekj)(\thekj-\thedag\),\\[0.5ex]
E_2&\leq 2\|L\|\epk K_R\|LS(\thekj)-LS(\thedag)\|=2\|L\|\epk K_R\|LS(\thekj)-y\|,\\[0.5ex]
E_3&\leq (1+\epsilon)M_R^2\|LS(\thekj)-y^\delta\|^2,\\[0.5ex]
E_4&\leq \(1+\frac{1}{\epsilon}\)M_R^2\|L\|^2 \epk^2.
\end{align*}
Combining these estimates then inserting $y$ and applying tangential cone condition in the equivalent form \eqref{cond-upper-tcc1}, we get
\begin{align*}
&\etildk\leq \big(E_1+E_3\big)+\big(E_2+E_4\big)\\
&\leq -\big(2-(1+\epsilon)M_R^2\big)\|LS(\thekj)-y^\delta\|^2
-2\( LS(\thekj)-y^\delta,y^\delta-LS(\thekj)+LS'(\thekj)(\thekj-\thedag) \)\\
&\hspace{1cm}+\|L\|\epsilon_{K(j)}\Big[ 2K_R\|LS(\thekj)-y\|+\(1+\frac{1}{\epsilon}\)M_R^2\|L\| \epk\Big]\\
&\leq -\big(2-(1+\epsilon)M_R^2\big)\|LS(\thekj)-y^\delta\|^2+ 2\|LS(\thekj)-y^\delta\|\( \delta+C_{tc}(\delta+\|LS(\thekj)-y^\delta\|) \)\\
& \hspace{1cm} +\|L\|\epk\Big[ 2K_R(\delta+\|LS(\thekj)-y^\delta\|) +\(1+\frac{1}{\epsilon}\)M_R^2\|L\| \epk) \Big]\\
&=\|LS(\thekj)-y^\delta\| \Big[ -\Big(2-(1+\epsilon)M_R^2-2C_{tc}\Big)\|LS(\thekj)-y^\delta\| + 2\delta(1+C_{tc})\Big]\\
& \hspace{1cm} +\|L\|\epk\Big[ 2K_R\delta+2K_R\|LS(\thekj)-y^\delta\|) +\(1+\frac{1}{\epsilon}\)M_R^2\|L\| \epk) \Big].
\end{align*}
This estimate reveals that if the state approximation error $\epk$ is sufficiently small in comparison to $\|LS(\thekj)-y^\delta\|$, we obtain negativity of $\etildk$. This motivates us to control the lower-level error $\epk$ by the upper-level residual $\gamma(j)\|LS(\thekj)-y^\delta\|$ as in \eqref{cond-lower-stopindex}. Indeed, \eqref{cond-lower-stopindex} leads us to
\begin{align*}
\etildk\leq &\|LS(\thekj)-y^\delta\|\,\cdot\\& \Bigg[ -\Big(2-(1+\epsilon)M_R^2-2C_{tc}- \|L\|\( 2K_R{\gamma(j)}+M_R^2\|L\|{\gamma(j)^2} \)\big(1+\frac{1}{\epsilon}\big) \Big)\|LS(\thekj)-y^\delta\|\\[1ex]
&\hspace{5cm}+2\delta\big(1+C_{tc}+\|L\|K_R\gamma(j)\big)  \Bigg]\quad <0,
\end{align*}
where negativity is achieved if we stop the upper-level iteration according to the rule \eqref{upper-posterior-choice}. Fej\'er monotonicity follows.
\end{proof}

In view of \eqref{upper-posterior-choice}, we suggest the posterior upper-level stopping rule
\begin{equation}\label{upper-stopindex}
\boxed{\|LS(\widetilde{\theta}^{\delta,\jmax})-y^\delta\|\leq\tau \delta< \|LS(\widetilde{\theta}^{\delta,j})-y^\delta\|, \quad \tau > \Gamma(j), \quad 0\leq j<\jmax 
}
\end{equation}
with $\Gamma(j)$ as in \eqref{upper-posterior-choice}. By Fej\'er monotonicity, $\thekjn$ is a better approximation of $\theta^\dagger$ than $\thekj$, as long as the discrepancy \eqref{upper-stopindex} is respected.

\begin{remark}[Stopping rule -- consistency check]
We wish to emphasize that when $S$ is solved exactly or an analytic solution of the PDE \eqref{original-pde} can be derived, then in \eqref{upper-posterior-choice} one can set $K(j)=\infty$, $\gamma(j)= \epsilon=0$ and bypass the lower-level problem. In this event, \eqref{upper-posterior-choice} reduces to  
\[
\|LS(\thekj)-y^\delta\|>2\dfrac{1+C_{tc}}{1-2C_{tc}}\,\delta \qquad \text{for }M_R=1, C_{tc}<\frac{1}{2},
\]
coinciding with the known discrepancy principle in \cite[Section 2.2, (2.12)]{KalNeuSch08}.
\end{remark}


\subsection{A priori stopping rule}

At this point, it is natural to ask how one may obtain the exact evaluation $S(\thekj)$ required to compare the residual $\|LS(\thekj)-y^\delta\|$ with $\tau\delta$ as the discrepancy principle \eqref{upper-stopindex} suggests, and how terminate the lower-level iteration w.r.t the upper-level residual \eqref{cond-lower-stopindex}. A possible scenario is that one may have access to the system output $LS:\thekj\mapsto LS(\thekj)$, despite $S(\thekj)$ not being directly accessible. In this event, the proposed posterior stopping rule \eqref{upper-stopindex} can be verified. If this is not the case, an a priori stopping rule $\jmax$ via the noise level $\delta$ is desirable. 

\begin{proposition}[Uniform boundedness]\label{propprior-upper-uniformbound}
Let Assumptions \ref{summary-ass-bi}, \ref{summary-ass-low}, \ref{summary-ass-up} hold. All the iterates $\thekj$ with $0\leq j\leq\jmax(\delta)$ remain in the ball $\Bthe$ if, for some $\gamma(j)\geq 0$, the two following conditions are satisfied:

\begin{enumerate}

\item The lower-level problem \eqref{al:approx-S} is iterated until step $K(j)$, such that the state approximation error satisfies
\begin{align}\label{cond-prior-lower-stopindex}
\epk= \mathcal{O}\(\frac{1}{K(j)^{1/(2\alpha)}}\)\leq \delta\gamma(j)
\end{align}

\item In the upper-level, the initialization and the prior stopping index satisfy
\begin{equation}\label{upper-prior-choice}
\begin{split}
&\|\theta^0-\theta^\dagger\|\leq R_0<R,\\[1ex]
&\jmax(\delta)\delta^2\(\frac{K_R^2}{\mu_R}+5M_R^2+ 2\frac{K_R^2}{\mu_R}\|L\|^2\gamma(j)^2 +\frac{7}{2}M_R^2\|L\|^2\gamma(j)^2\)\leq R^2-R_0^2.
\end{split}
\end{equation}

\end{enumerate}

\end{proposition}

\begin{proof}
Similarly to Proposition \ref{prop-upper-Fejer}, we will evaluate the error $\etildk$ between two consecutive iterations. We begin with the estimate \eqref{auxi-esstimate}, where in the last line, we further insert $y$ into $E_1$ and $E_3$ (that is, the terms containing $y^\delta$). So
\begin{align*}
&\etildk=\|\thekjn-\thedag\|^2-\|\thekj-\thedag\|^2\leq A+B.
\end{align*}
where $A$ has the same form as as $(E_1+E_2+E_3+E_4)$ in \eqref{auxi-esstimate}, with $y$ replacing $y^\delta$ in $E_1$, $E_3$, and $B$ is the residual generated by inserting $y$.

In $A$, we can directly apply the first part of the tangential cone condition \eqref{cond-upper-tcc} to $E_1$; the rest is estimated in a similar manner as in \eqref{auxi-esstimate}. That is,
\begin{align*}
A&=-2\( LS(\thekj)-y,LS'(\thekj)(\thekj-\thedag) \)-2\( L\Stil(\thekj)-LS(\thekj),LS'(\thekj)(\thekj-\thedag) \) \nonumber\\
&+(1+\epsilon)\|S'(\thekj)^*L^*(LS(\thekj)-y)\|^2+\(1+\frac{1}{\epsilon}\)\|S'(\thekj)^*L^*(L\Stil(\thekj)-LS(\thekj))\|^2\\
&\leq-2(M_R^2+\mu_R)\|LS(\thekj)-y\|^2+2\|L\|\epk K_R\|LS(\thekj)-y\|\\
&+(1+\epsilon)(1+\epsilon')M_R^2\|LS(\thekj)-y\|^2+\(1+\frac{1}{\epsilon}\) M_R^2 \|L\|^2\epk^2\\
&=: A_1+A_2+A_3+A_4.
\end{align*}
Noticing that in comparison to $E_3$, $A_3$ has the extra constant $(1+\epsilon')$ generated from splitting the square term $(a+b)^2\leq (1+\epsilon')a^2+(1+\frac{1}{\epsilon'})b^2$, $\epsilon'>0$ when inserting $y$.

In $B$, we use the second part of the tangential cone condition \eqref{cond-upper-tcc}, boundedness of the derivative as in \eqref{cond-upper-boundedG'}, then finally Young's inequality $2ab\leq \mu_R a^2+b^2/\mu_R$ for $a:=\|LS(\thekj)-y\|$, $b:=\delta K_R$. This leads to
\begin{align*}
B&= -2\( y-y^\delta,LS'(\thekj)(\thekj-\thedag) \)+ (1+\epsilon)(1+\frac{1}{\epsilon'})\|S'(\thekj)^*L^*(y-y^\delta)\|^2\\
&\leq 2\delta K_R\|LS(\thekj)-y\|+(1+\epsilon)(1+\frac{1}{\epsilon'})\delta^2 M_R^2\\
&\leq \mu_R\|LS(\thekj)-y\|^2+\delta^2\(\frac{K_R^2}{\mu_R}+(1+\epsilon)(1+\frac{1}{\epsilon'})M_R^2\)\\
&=:B_1+B_2
\end{align*}
We now set $\epsilon=\epsilon'=\sqrt{2}-1$ in both $A$ and $B$. In $A_3$, we thus have the constant $(1+\epsilon)(1+\epsilon')=2$; in $A_4$, we have $(1+\frac{1}{\epsilon})=\sqrt{2}/(\sqrt{2}-1)<7/2$, while in $B_2$, we have $(1+\epsilon)(1+\frac{1}{\epsilon'})<5$.

Next, we bound $\epk$ in $A_2$, $A_4$ by $\delta \gamma(j)$ as the rule \eqref{cond-prior-lower-stopindex} suggests. In $A_2$, we further apply Young's inequality $2ab\leq \mu_Ra^2/2 +2b^2/\mu_R$ for $a:=\|LS(\thekj)-y\|$, $b:=\|L\|\epk K_R$. 

Now summing $A$ and $B$, we see that $(A_3+B_1+\mu_Ra^2/2)$ is absorbed by part of $A_1$, that is
\begin{align}\label{upper-monontone}
\etildk&\leq A +B \leq \Big(A_1+A_3+B_1+\mu_Ra^2/2\Big) + \Big(B_2+2b^2/\mu_R+A_4\Big)\nonumber\\
&\leq -\frac{\mu_R}{2}\|LS(\thekj)-y\|^2+\delta^2\(\frac{K_R^2}{\mu_R}+5M_R^2+ 2\frac{K_R^2}{\mu_R}\|L\|^2\gamma(j)^2 +\frac{7}{2}M_R^2\|L\|^2\gamma(j)^2\).
\end{align}

Next, summing up $\etildk$ for $j=0\ldots\bar{j}$ for any $\bar{j}\leq\jmax(\delta)$ with the upper-stopping index $\jmax$ proposed in \eqref{upper-prior-choice}, we arrive at
\begin{align}\label{upper-suming}
\frac{\mu_R}{2}&\sum_{j=0}^{\bar{j}}\|LS(\thekj)-y\|^2+\|\thekjbar-\theta^\dagger\|^2 \nonumber\\
&\leq \|\theta^0-\theta^\dagger\|^2+\jmax(\delta)\delta^2\(\frac{K_R^2}{\mu_R}+5M_R^2+ 2\frac{K_R^2}{\mu_R}\|L\|^2\gamma(j)^2 +\frac{7}{2}M_R^2\|L\|^2\gamma(j)^2\)\nonumber\\
&\leq R^2.
\end{align}
This shows that all the iterates until $\jmax(\delta)$ remain in the ball
\[\thekj\in \Bthe \quad \forall j\leq\jmax(\delta),\]
and the proof is complete.
\end{proof}

In comparison to Proposition \ref{prop-upper-Fejer}, a prior stopping rule cannot yield Fej\'er monotonicity of the iterate with noisy data. As monotonicity is essential in ensuring the regularizing property of the Landweber iteration, and also due to the bi-level factor, we need to take particular care regarding stability of the upper-level. 

The following Proposition is crucial to derive a more concrete stopping rule, 
which can ensure controllability of the approximation errors propagating to the upper-level. At this point, the preliminary error analysis in Section \ref{sec:pre-erroranalysis} comes into play. We remind ourselves that the circumflex $\widetilde{(\cdot)}$ denotes the approximation/bi-level scheme; notation without the circumflex refers to the exact/single-level scheme.

\begin{proposition}[Noise propagation-approximation error]\label{prop-upper-noisepropagate}
Let the assumption in Proposition \ref{propprior-upper-uniformbound} hold. Assume moreover that the lower-level precision \eqref{cond-prior-lower-stopindex} satisfies
\begin{align}\label{condupper-lowerstop}
\epk\leq\delta\gamma(j)\leq \frac{\delta}{q^j} \qquad \text{for some }q\geq1.
\end{align}
Then, we achieve the estimate
\begin{align}\label{upper-estimate}
&\|\thekj-\theta^j\| \leq  \delta\,\hat\Gamma(j), \\
& \hat{\Gamma}(j):=M_R\frac{\Big[1+M_R^2+(1+M_S)\mathbb{D} \Big]^j-1}{M_R^2+(1+M_S)\mathbb{D}}
+\frac{\(M_R\|L\|+ \mathbb{D}\)}{q} \frac{ \Big[1+M_R^2+(1+M_S)\mathbb{D} \Big]^j-(1/q)^j}{\Big[1+M_R^2+(1+M_S)\mathbb{D} \Big]-(1/q) } \nonumber
\end{align}
for all $j\leq\jmax(\delta)$, where 
\begin{align*}
\delta\leq\bar{\delta}, \qquad\gamma(j)\leq \bar{\gamma},\qquad
\mathbb{D}:=C_{\nabla f,A}\|L\|(\bar{\delta}+M_R R+\|L\|\bar{\delta}\bar{\gamma}),\qquad C_{\nabla f,A} \text{ as in } \eqref{error-adjoint}.
\end{align*}
\end{proposition}
\begin{proof} 
We begin with comparing the $(j+1)$-iteration of the single-level scheme with exact data, and of the bi-level scheme with noisy data
\begin{align*}
&\theta^{j+1} = \theta^j -S'(\theta^j)^*L^*\(LS(\theta^{j}-y\),\\
&\thekjn = \thekj -S'(\thekj)^*L^*\(L\Stil(\thekj)-\ydel\).
\end{align*}
Inserting $S'(\theta^j)^*L^*(L\Stil(\thekj)-\ydel)$ (first estimate below) and using the analysis of the state approximation error derived in Lemma \ref{lem-outputerror} (second and third estimates) and of the adjoint approximation  error derived Lemma \ref{lem-adjointerror} (second estimate), we deduce
\begin{align*}
&\|\theta^{j+1}-\thekjn\|\\&\leq\|\theta^j-\thekj\|\\
&\quad+\|S'(\theta^j)^*L^*\(LS(\theta^j)-L\Stil(\thekj)\)\|+\|S'(\theta^j)^*L^*\|\delta\\
&\quad+\|\(S'(\theta^j)^*-S'(\thekj)^*\)L^*\(L\Stil(\thekj)-\ydel\)\|\\
&\leq\|\theta^j-\thekj\| \\
&\quad+ M_R\(M_R\|\theta^j-\thekj\|+\|L\|\epk\)+M_R\delta\\
& \quad+ C_{\nabla f,A}\((1+M_S)\|\theta^j-\thekj\|+ \epk\)\|L\|\(\delta+\|L\Stil(\thekj)-LS(\theta^\dagger)\|\)\\
&\leq\|\theta^j-\thekj\| \\&\quad+ M_R\(M_R\|\theta^j-\thekj\|+\|L\|\epk\)+M_R\delta\\
& \quad+ C_{\nabla f,A}\((1+M_S)\|\theta^j-\thekj\|+ \epk\)\|L\|\(\delta+M_R\|\thekj-\theta^\dagger\|+\|L\|\epk\).
\end{align*}
Simplify the notation by denoting 
\begin{align*}
&T^j:=\|\theta^j-\thekj\|,\\
&\|L\|\(\delta+M_R\|\thekj-\theta^\dagger\|+\|L\|\epk\)\leq\|L\|(\delta+M_R R+\|L\|\epk)=:B_{\delta,\epk},
\end{align*}
where in $B_{\delta,\epk}$, we have invoked uniform boundedness of the iterates in the ball $\Bthe$ proved in Proposition \ref{propprior-upper-uniformbound}. We now group the estimate in terms of $T^j, \delta$ and $\epk$, and obtain
\begin{align}\label{upper-induction}
&T^{j+1}\\&\quad\leq T^j \Big[1+M_R^2+C_{\nabla f,A}(1+M_S)B_{\delta,\epk} \Big] + \delta M_R  +\epk\Big[ M_R\|L\|+ C_{\nabla f,A}B_{\delta,\epk} \Big].\nonumber
\end{align}

The estimate \eqref{upper-induction} exposes an important fact: if $\epk$ decays sufficiently fast with respect to $\delta$ 
then we can bound the term $T^{j+1}$ from above. Indeed, \eqref{cond-prior-lower-stopindex} suggests running the lower-level iterate until $\epk\leq\delta\gamma(j)$. 
Employing the upper bounds $\bar{\delta}$ and $\bar{\gamma}$ from the statement of the proposition, we have 
\[C_{\nabla f,A}B_{\delta,\epk}\leq C_{\nabla f,A}\|L\|(\bar{\delta}+M_R R+\|L\|\bar{\delta}\bar{\gamma})=:\mathbb{D}.\]
This results in the recursive relation
\begin{align}
T^{j+1}\leq T^j \Big[1+M_R^2+(1+M_S)\mathbb{D} \Big] + \delta M_R +  \delta\gamma(j) \Big[ M_R\|L\|+ \mathbb{D}\Big].
\end{align}

By induction and choosing the same initialization $\widetilde{\theta}^{\delta,0}=\theta^{0}$, i.e. $T^0=0$, we arrive at the important estimate
\begin{align}\label{upper-induction1}
&T^j\leq  \delta\Bigg( M_R\frac{\Big[1+M_R^2+(1+M_S)\mathbb{D} \Big]^j-1}{M_R^2+(1+M_S)\mathbb{D}}\\
&\qquad\qquad\qquad+\( M_R\|L\|+ \mathbb{D}\) \sum_{i=1}^j \Big[1+M_R^2+(1+M_S)\mathbb{D} \Big]^i\gamma(j-i)\Bigg) \qquad\text{for all }j\leq\jmax(\delta).\nonumber
\end{align}
By \eqref{condupper-lowerstop}, $\gamma(j)\leq\frac{1}{q^j}$; we can thus collapse the sum in \eqref{upper-induction1}, and arrive at the key estimate
\begin{align*}
T^j=\|\theta^j-\thekj\| \leq  \delta\hat{\Gamma}(j)
\end{align*}
for all $j\leq\jmax(\delta)$, as required.
\end{proof}

This immediately suggests a stopping rule:
\begin{assumption}[prior stopping rule]
The upper-stopping index $\jmax(\delta)$ satisfies
\begin{equation}\label{propprior-upper-converge}
\begin{split}
& \delta \,\hat{\Gamma}\left(\jmax(\delta)\right) \to 0 \text{ as $\delta\to 0$,}
\end{split}
\end{equation}
with $\hat{\Gamma}$ as in \eqref{upper-estimate}.
\end{assumption}


With all the auxiliary results prepared, we now state the main  result.

\begin{theorem}[Convergence of bi-level iteration]\label{theo:bi-priorstop}
Assuming that the following both hold:
\begin{enumerate}[label=(\alph*)]
\item The upper-level stopping index $\jmax(\delta)$ is chosen according to \eqref{upper-prior-choice} and \eqref{propprior-upper-converge}.
\item The lower-level stopping index $K(j)$ is chosen according to \eqref{condupper-lowerstop} at each $j\leq\jmax(\delta)$.
\end{enumerate}
Then, the bi-level iterates $\thekjstar$ in Algorithm \ref{algorithm} converge to a solution of the inverse problem \eqref{original-pde}-\eqref{original-measure} as $\delta\to 0$ .
\end{theorem}
\begin{proof}
We begin by claiming that for exact data $y$, the whole sequence $\theta^j$ run with the exact/single-level scheme converges in the weak sense to a solution of the inverse problem \eqref{original-pde}-\eqref{original-measure}. Indeed, looking at the proof of Proposition \ref{propprior-upper-uniformbound}, in particular the expression \eqref{upper-monontone}, by setting $\delta=0$, $\gamma(j)=0$ we have 
\[e_j^{\delta=0}=\|\theta^{j+1}-\thedag\|^2-\|\theta^j-\thedag\|^2<0 \qquad j\in\N\]
showing Fej\'er monotonicity. \eqref{upper-suming} with $\bar{j}=\infty$, $G=L\circ S$ becomes
\begin{align*}
\frac{\mu_R}{2}&\sum_{j=0}^\infty\|G(\theta^j)-y\|^2\leq R^2.
\end{align*}
These two properties are exactly the same as \eqref{low-Fejer} and \eqref{low-sum-F} in Theorem \ref{low-conv-rate}, the key steps to prove convergence of the iteration, with $\theta$ in place of $u$, and $G$ in place of $\Ft$. Therefore, following the line of the proof of Theorem \ref{low-conv-rate}, part i), noting that the tangential cone condition (even the strong version) is also assumed here, we can assert that the whole iterate sequence of the single-level algorithm with noise-free data converges to a solution to the inverse problem problem, i.e.
\begin{align*}
\theta^j\overset{j\to\infty}{\rightharpoonup} \theta^\dagger \qquad\text{with}\qquad G(\theta^\dagger)=y.
\end{align*}
Moreover, one can claim strong convergence. This is done in \cite[Theorem 2.4]{KalNeuSch08} (which we do not repeat here) by showing that $e_j^{\delta=0}$ in fact forms a Cauchy sequence; thus,
\begin{align}\label{noisefree-conv}
\theta^j\overset{j\to\infty}{\to} \theta^\dagger \qquad\text{with}\qquad G(\theta^\dagger)=y.
\end{align}

This, together with Proposition \ref{prop-upper-noisepropagate} and the stopping rule \eqref{propprior-upper-converge}, allows us to conclude convergence of the sequence $\thekjsub$ via the decomposition
\begin{align*}
\thekjsub-\theta^\dagger = \big(\thekjsub - \theta^{\jmax(\delta_n)}\big) + \big(\theta^{\jmax(\delta_n)}-\theta^\dagger\big) \, \overset{n\to\infty}{\to} \, 0, 
\end{align*}
where convergence of the first term follows from Proposition \ref{prop-upper-noisepropagate} with the rule \eqref{propprior-upper-converge}, and convergence of the second term is  \eqref{noisefree-conv}. The proof ends here.

\end{proof}

\section{Discussion on assumptions and application}\label{sec:discus}
\subsection{Classes of problems with coercivity}\label{sec:verify-coercive}
As presented in  Remark \ref{rem-coerivity} of Section \ref{sec:lowiter}, coercivity \eqref{low-rate} is the essential property that we extract from well-posedness of $S$ to achieve the convergence rate for the lower-level iterations. 
In what follows, we present two classes of PDEs where coercivity can be established: Lipschitz nonlineartiy and monotone-type nonlinearity \eqref{app-3}.

Let us consider the PDE \eqref{low-IP}, where the unknowns  $\theta:=(a,c,\phi,u_0)\in X$ consist of the diffusion coefficient $a$, potential $c$, source term $\phi$ and initial state $u_0$. These unknown parameters are embedded in the PDE model via
\begin{align}\label{app-1}
\Ft(u):\Vc\mapsto\Uc^*\times H\qquad \Ft(u)=
\begin{pmatrix}
\dot{u}-\nabla\cdot(a\nabla u)+ c u + \Phi(u)- \phi\\
u_0-u|_{t=0}
\end{pmatrix}
\end{align}
in the Hilbert space framework (c.f. Section \ref{sec:setting})
\begin{equation}\label{app-2}
\begin{split}
&U=H^1_0(\Omega), \quad\qquad H=L^2(\Omega), \qquad U^*=H^{-1}(\Omega),\nonumber\\[0.5ex]
&\Vc = L^2(0,T;U)\cap H^1(0,T;U^*), \qquad\Uc^* = L^2(0,T;H^{-1}(\Omega),\\[0.5ex]
&X=(X_a,X_c,X_\phi,X_0),\\
&X_a=\{a\in H^2(\Omega):a\geq\underline{a}>0\},\qquad  X_c=X_0=L^2(\Omega),\qquad X_\phi=H^{-1}(\Omega)
\end{split}
\end{equation}
with bounded Lipschitz domain $\Omega\subset\R^3$.

In \eqref{app-1}, the nonlinearity $\Phi:\Vc\to\Uc^*$ is a Nemytskii operator induced by $\Phi:U\to U^*$ with (by abuse of notation) $[\Phi(u)](t):=\Phi(u(t))$. We consider
\begin{align}\label{app-3}
\Phi \begin{cases}
 \text{either (a) Lipschitz continuous:} &\|\Phi(u)-\Phi(v)\|_{U^*}\leq L_\Phi\|u-v\|_H,\\[1ex]
\text{or (b) of \emph{monotone-type}:} &\langle \Phi(u)-\Phi(v),u-v\rangle_{U^*,U} +M_\Phi\|u-v\|_H^2\geq 0
\end{cases}
\end{align}
for all $u$, $v\in U$ with $L_\Phi$, $M_\Phi\geq0$. Clearly, case (b) covers the class of standard monotonic functions.

\begin{proposition}[Coercivity verification]\label{prop-PDEcoercive}
The classes of PDEs with Lipschitz or monotone-type nonlinearities \eqref{app-1}-\eqref{app-3} satisfy the coercivity property listed in Assumption \ref{summary-ass-low} with $\alpha=1$ and in $\Uc$-norm. That is, the relation
\begin{align}\label{lipschitz-stability}
C_{coe}\|\Ft(u)-\Ft(v)\|_{\Uc^*\times H}\geq \|u-v\|_{\Uc} \qquad \forall u,\,v\in\Vc
\end{align}
holds uniformly for all $\theta:=(a,c,\phi,u_0)\in\Bthe$ with a positive constant $C_{coe}$. 
\end{proposition}
The property \eqref{lipschitz-stability} is also referred to as \emph{Lipschitz stability}.
\begin{proof}
We begin by testing $\Ft(u)-\Ft(v)$ with $(u-v; 0)\in\Vc\times H$. Setting $\|u\|_{H_0^1}:=\|\nabla u\|_{L^2}$ due to the Gagliardo–Nirenberg–Sobolev inequality and recalling $a\geq \underline{a}>0$, we write
\begin{align}\label{energy}
&\langle \Ft(u)-\Ft(v),\(u-v; 0\) \rangle_{U^*\times H,U\times H} \nonumber\\
&=\langle \dot{u}-\dot{v}-\nabla\cdot(a \nabla(u-v))+ c (u-v) + \Phi(u)-\Phi(v),u-v\rangle_{U^*,U}
\\
&\geq\frac{1}{2}\frac{d}{dt}\|u-v\|^2_{L^2} + \underline{a}\|u-v\|_{H^1}^2 + \langle c(u-v),u-v\rangle_{U^*,U}+\langle \Phi(u)-\Phi(v),u-v\rangle_{U^*,U}\nonumber
\end{align}
Estimating the $c$-term by using H\"older's inequaltiy with $p=q=2$, then by Young's inequality $xy\leq x^p/p+y^q/q, p=4/3,q=4$, we get
\begin{align*}
&|\langle c(u-v),u-v\rangle_{U^*,\,U}|\\
&\qquad\leq  \|c\|_{L^2}\|(u-v)^{3/2}(u-v)^{1/2}\|_{L^2}\leq (C_{H^1\to L^6}^\Omega)^{3/2}\|c\|_{L^2}\|u-v\|^{3/2}_{H^1}\|u-v\|^{1/2}_{L^2} \\
&\qquad\leq \frac{3}{4}\(\underline{a}^{3/4}\|u-v\|^{3/2}_{H^1}\)^{4/3}+\frac{1}{4} \(\frac{(C_{H^1\to L^6}^\Omega)^{3/2}}{\underline{a}^{3/4} }\|c\|_{L^2}\|u-v\|^{1/2}\)^4\\
&\qquad=\frac{3\underline{a}}{4}\|u-v\|^2_{H^1}+\frac{(C_{H^1\to L^6}^\Omega)^6}{4\underline{a}^3}\|c\|^4_{L^2}\|u-v\|^2_{L^2},
\end{align*}
where $C_{H^1\to L^6}^\Omega$ indicates the application of the continuous embedding $H^1(\Omega)\embed L^6(\Omega), \Omega\subset\R^3$.
For the nonlinear term $\Phi$ in \eqref{app-3}, one observes
\begin{align*}
\langle\Phi(u)-\Phi(v),u-v\rangle
\begin{cases}
\leq \frac{\underline{a}}{8}\|u-v\|^2_{H^1}+\frac{8(L_\Phi)^2}{\underline{a}}\|u-v\|^2_{L^2} \qquad &\text{case (a)}\\
\geq -M_\Phi\|u-v\|_{L^2}^2 &\text{case (b)}
\end{cases}
\end{align*}
where for the case (a), we have used Young's inequality $xy\leq x^2/\epsilon+\epsilon y^2, \epsilon=\underline{a}/8$. Combining the above estimates of the nonlinearity $\Phi$, the $c$-term to \eqref{energy}, we have that, for all $t\in[0,T]$
\begin{align}\label{energy1}
&\frac{1}{2}\frac{d}{dt}\|u(t)-v(t)\|^2_{L^2} + \frac{\underline{a}}{8}\|u-v\|_{H^1}^2 \leq \langle \Ft(u)-\Ft(v),\(u-v; 0\)\rangle_{U^*\times H,U\times H} +C_{\Phi,R}\|u-v\|^2_{L^2} \nonumber\\
&\frac{d}{dt}\(\frac{1}{2}\|u-v\|^2_{L^2}\) + (\frac{\underline{a}}{8}-\epsilon)\|u(s)-v(s)\|_{H^1}^2\leq \frac{1}{\epsilon}\|\Ft(u)-\Ft(v)\|_{H^{-1}}^2+C_{\Phi,R}\|u-v\|^2_{L^2},
\end{align}
where in the second line, we have applied one more time Young's inequality $ab\leq \epsilon a^2+b^2/\epsilon, \epsilon<\underline{a}/8$ and grouped the L$^2$-terms with the nonnegative constant \[C_{\Phi,R}:=(C_{H^1\to L^6}^\Omega)^6\|c\|^4_{L^2}/(4\underline{a}^3)+8(L_\Phi)^2/\underline{a}+M_\Phi.\]

In the second step, taking the integral of \eqref{energy1} over the time interval $(0,t)$, $t\leq T$ leads to
\begin{align}\label{gronwall}
&\(\frac{\underline{a}}{8}-\epsilon\)\|u-v\|_{L^2(0,t;H^1(\Omega))}^2+\frac{1}{2}\|u(t)-v(t)\|^2_{L^2} \nonumber\\
&\qquad\leq  \frac{1}{\epsilon}\|\Ft(u)-\Ft(v)\|_{L^2(0,t;H^{-1}(\Omega))}^2+\frac{1}{2}\|(u-v)|_{t=0}\|^2_{L^2}+C_{\Phi,R}\int_0^t\|u(s)-v(s)\|^2_{L^2}\,ds  \nonumber\\
&\qquad\leq  \(\frac{1}{\epsilon}+\frac{1}{2}\)\|\Ft(u)-\Ft(v)\|_{\Uc^*\times H}^2+C_{\Phi,R}\int_0^t\|u(s)-v(s)\|^2_{L^2}\,ds.
\end{align}
Now, we apply Gr\"onwall's inequality in the form \[x(t)\leq \Pi+C_{\Phi,R}\int_0^t x(s)\,ds\quad\Rightarrow\quad x(t)\leq \Pi + C_{\Phi,R}\int_0^t\Pi \exp^{C_{\Phi,R}(t-s)}\,ds \quad\forall t\in[0,T]\] 
for \eqref{gronwall}, where $x(t):=\frac{1}{2}\|u(t)-v(t)\|^2_{L^2}$ is the second term on the left hand side and $\Pi:=\(\frac{1}{\epsilon}+\frac{1}{2}\)\|\Ft(u)-\Ft(v)\|_{\Uc^*\times H}^2$. We then get
\begin{align*}
\|u-v\|_{C(0,T;L^2(\Omega))}^2\leq  \(\frac{2}{\epsilon}+1+2C_{\Phi,R}T\exp^{2C_{\Phi,R}T}\)\|\Ft(u)-\Ft(v)\|_{\Uc^*\times H}^2.
\end{align*}
In \eqref{gronwall}, one can moreover bound the term $(\underline{a}/8-\epsilon)\|u-v\|_{L^2(0,T;H^1(\Omega))}$ on the left hand side by the entire right hand side, thus eventually by $\|\Ft(u)-\Ft(v)\|_{\Uc*\times H}$. Altogether, we obtain
\begin{align}\label{verify-coercivity}
\|u-v\|_{C(L^2)\cap\, \Uc}\leq  \widetilde{C}_{\Phi,R}\|\Ft(u)-\Ft(v)\|_{\Uc^*\times H}
\end{align}
with some $\widetilde{C}_{\Phi,R}>0$. Importantly, as the parameter $\theta$ lies in the ball $\Bthe$ with finite radius $R$, the estimate \eqref{verify-coercivity} is locally uniform. If only $(\varphi,u_0)$ are the unknown parameters, the radius $R$ is allowed to be arbitrarily large, as ${C}_{\Phi,R}$, thus $\widetilde{C}_{\Phi,R}$, depends only on $(c,a)$. This shows Lipschitz stability of $u\mapsto\Ft(u)$, equivalently the desired coercivity.
\end{proof}

\begin{remark}[coercivity with $\Vc$-norm]\label{rem-PDEcoercive-V}
By evaluating \[\dot{u}-\dot{v}=\Delta (u-v)- c (u-v) - \Phi(u)+\Phi(v)\] in $L^2(0,T;H^{-1}(\Omega))$-norm, one can from the preceding result immediately deduce the stronger estimate
\begin{align*}
\|u-v\|_{\Vc}\leq  \widetilde{C}_{\Phi,R}\|\Ft(u)-\Ft(v)\|_{\Uc^*\times H}
\end{align*}
for Lipschitz nonlinearity $\Phi$, case (a) in \eqref{app-3}. However, as pointed out in Remark \eqref{rem-derivative}-iii), a $\Uc$-norm  coercivity is sufficient for the bi-level algorithm.
\end{remark}

\begin{remark}[Tangential cone condition]\label{dis:TCC} In view of Corollary \ref{coerive-alpha=1}, the tangential cone condition for the lower-level problem (in $u$) is clearly satisfied for $\Phi(u)=0$. 
When $\Phi(u)$ is as in \eqref{app-3}, as coercivity/Lipschitz stability property is already establish in Proposition \ref{prop-PDEcoercive}, one only needs to assume Lipschitz continuity of its derivative to conclude the tangential cone condition. Thus, convergence of the lower-level iteration is highly achievable.

For the upper-level problem, verification of the tangential condition (in $\theta$) is another technical question. We refer to \cite{TCC21} for a study of such condition in several linear and nonlinear parabolic benchmark problems, to \cite{HoffmanWaldNguyen:2021} for abstract linear elliptic problems, to \cite{HubmerScherzer:TCC18} for quantitative elastography, to \cite{Nakamura:MRE21} for magnetic resonance elastography, to \cite{Kindermann21} for the electrical impedance tomography, to \cite{Rieder:TCC21} for full wave from inversion, 
and recently, to \cite{holler22learning_parameter_id}, \cite{ScherzerHofmannNashed}  in the context of machine learning.
\end{remark}

\subsection{Application to magnetic particle imaging}\label{sec:mpi}
This section is dedicated to discussion regarding a real-life application:  magnetic particle imaging (MPI). 
The material of the following discussion is extracted from the two papers \cite{KNSW, NguyenWald:2022}. The work in \cite{KNSW} puts forward the mathematical analysis of the parameter identification problem in MPI; subsequently, \cite{NguyenWald:2022} proceeds with Landweber, Landweber-Kaczmarz algorithms and numerical experiments. Moreover, an ad hoc bi-level algorithm, in which lower-level iteration is used as an approximation method for the nonlinear PDE \eqref{mpi-llg}, was successfully implemented without full analysis.

\paragraph{State of the art.}
MPI is a novel medical imaging technique invented by B. Gleich and J. Weizenecker in 2005 \cite{tktb12}. The underlying concept is observing the magnetic behavior of iron oxide nanoparticle tracers, which are injected into patients and respond to an external oscillating magnetic field (with a field-free-point). One can thus visualize the particles' concentration, yielding quantitative imaging of blood vessels. MPI offers favorable features: no harmful radiation, high resolution ($<1$ mm) and fast acquisition ($< 0.1$ s). 


\paragraph{Parameter identification in MPI.}
The external magnetic field induces the temporal change $\partial\mathbf{m}/\partial t$ of the particles, which by Faraday's law induces a measurable electric voltage 
\begin{equation}\label{mpi-obs}
 v(t) = \int_0^T \int_{\Omega} -\mu_0 \widetilde{a}_l(t-\tau) c(x) \, \mathbf{p}^{\mathrm{R}}(x) \cdot \frac{\partial}{\partial \tau} \mathbf{m}(x,\tau) \, \mathrm{d} x \, \mathrm{d}=: \mathcal{K}\mv.
\end{equation}
Here, $\mu_0$ is magnetic permeability in vacuum, $\widetilde{a}_l$ is a band-pass filter, $c$ is the known particle concentration and $\mathbf{p}^{\mathrm{R}}$ is the 3D receive coil sensitivity.

 Inspired by \cite{tlg04, llel92}, we use the Landau-Lifshitz-Gilbert (LLG) equation to model the evolution of the magnetization moment vector $\mv$ in a microscopic scale
\begin{align}\label{mpi-llg}
\alpha_1 m_{\mathrm{S}}^2 \frac{d}{dt}\mathbf{m} - \alpha_2 \mathbf{m} \times \frac{d}{dt} \mathbf{m}-m_{\mathrm{S}}^2\Delta \mathbf{m} &= \lvert \nabla \mathbf{m} \rvert^2 \mathbf{m} + m_{\mathrm{S}}^2\mathbf{h} - 
\langle \mathbf{m}, \mathbf{h} \rangle \mathbf{m} \ && \text{in} \ \Omega \times [0,T], \nonumber\\
 0 &= \partial_{\nu} \mathbf{m}  && \text{on} \ \partial\Omega \times [0,T], \\
 \mathbf{m}_0 &= \mathbf{m}(t=0), \ \lvert \mathbf{m}_0 \rvert = m_{\mathrm{S}} && \text{in} \ \Omega\nonumber
\end{align}
with the relaxation parameter $\alpha\in\R^2$, external magnetic field $\mathbf{h}$ and the saturation magnetization constant 
$m_S$. 
We highlight that as \eqref{mpi-llg} includes \emph{the relaxation effect} (via $\alpha$ and the cross product) and the \emph{particle interaction} (via $\Delta \mathbf{m}$ and $|\nabla \mathbf{m} \rvert^2 \mathbf{m}$), our approach is  realistic to a greater extent in comparison to simplified models, e.g. the Langevin function.

Combining \eqref{mpi-obs} and \eqref{mpi-llg}, we investigate the inverse problem of identifying $\alpha$ along side the magnetization $\mv$
\begin{align}\label{mpi-ip}
\text{Find }  \alpha=\highlight{(\alpha_1,\alpha_2)}:\quad \mathcal{K}\mv(\alpha)=v\qquad\text{s.t. }\quad \mv(\alpha) \text{ solves LLG }\eqref{mpi-llg}.
\end{align}
The LLG equation \eqref{mpi-llg} is highly nonlinear, with vectorial solution $\mv$; numerical solution of  \eqref{mpi-llg} is a greatly challenging task \cite{alougesKritsikisSteinerToussaint,BanasEtAl,BartelsProhl1, cimrak} (see also \cite[Remark 1]{NguyenWald:2022} for a survey). This initiated our idea of a bi-level approach, in which the lower-level problem is used to approximate $\mv$.

We invite the reader to review \cite{KNSW, NguyenWald:2022} for various numerical results. \highlight{To illustrate, the following Figure  shows the state approximation $\m(t,x)$ in the lower-level \cite[Fig.~9]{NguyenWald:2022} and the reconstruction as well as the internal operation of the bi-level scheme  \cite[Fig.~14]{NguyenWald:2022}. The reconstruction with Landweber step size $\mu=1$ ran 250 upper-level iterations and a total of 11095 lower-level iterations over 90,000 seconds. Note that this can be sped up by using the larger step size $\mu=10$. The final state error is below 1\%, and the parameter error is below 2\% \cite[Table 13]{NguyenWald:2022}.} In essence, the proposed bi-level scheme demonstrates its robustness through several test cases with experimental parameters, true physical parameters, noise free and noisy data. \highlight{We are currently carrying out an extended numerical comparison of the three approaches -- reduced, all-at-once and bi-level.}

\begin{figure}[!htb] 
\centering
\includegraphics[width=0.99\textwidth,keepaspectratio]{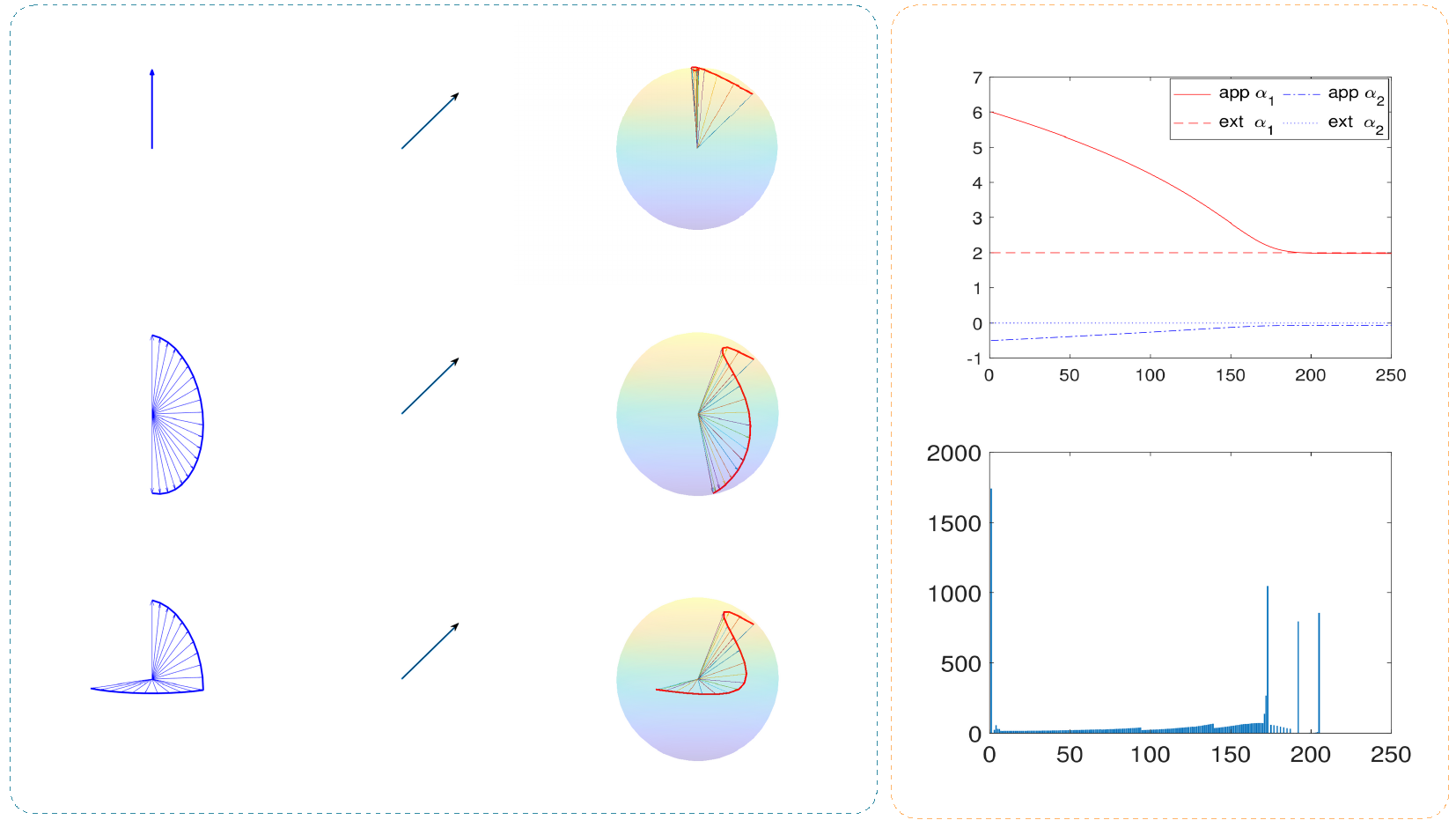}
\caption{\highlight{Left: Lower-level, applied field $\h(t)$ (left), initial state $\m_0$ (middle),  approximate trajectory $\m(t)$ (right). Top right: Bi-level convergence via reconstruction. Bottom right: Number of iterations for bi-level. 250 total upper iterations (x-axis), corresponding number of lower-iterations (y-axis). These figures are taken from \cite[Figs.~9, 14]{NguyenWald:2022}.}}
\end{figure}

\section{Conclusion}\label{sec:conclude}

In this work, we have proposed a bi-level Landweber framework with convergence guarantee for coefficient identification in ill-posed inverse problems. The method is tailored, but not limited, to nonlinear time-dependent PDE models. In addition, the regularization study in the upper-level can be combined with other inexact PDE solvers like reduced order models, principal component analysis etc. to obtain an end-to-end error analysis.

The author intends to extend the study in the following directions:
\begin{itemize}
\item The lower-level could further include a model inexactness  term $\delta_{PDE}$. This reflects the case, for instance, when the PDE is not modeled in ideal conditions, or when a simplified version is employed. In this situation, the PDE itself contains a quantity of inexactness, and the problem becomes ill-posed in both levels.
\item In Algorithm \ref{algorithm}, the lower-level iteration is initiated at some $u_0$ near $u^\dagger$, and the lower-level iterations between different upper-iterates do not communicate with each other. Arguing that a small update in the parameter $\theta$ causes a corresponding small update in the PDE solution $u$, one can initialize the lower-level iteration of step $j+1$ at the last iteration of step $j$, i.e. $u_0^{j+1}:=u^j_{K(j)}$, with the aim of lowering the stopping index $K(j+1)$. Inspired by the all-at-once approach, the author also wishes to investigate the situation where one lower-level iteration is sufficient, i.e. $K(j+1)=K(j)+1$, yielding alternating updates between $u$ and $\theta$.
\end{itemize}


Replacing the Landweber method by its accelerated version \cite{Neubauer17,HubmerRamlau17} or by another iterative method in either or both levels is also a promising research direction.

\paragraph{Acknowledgments.} The author thanks Barbara Kaltenbacher for many valuable comments, and Christian Aarset for thoroughly proofreading the manuscript. The author wishes to thank the reviewers for their insightful suggestions, leading to an improvement of the manuscript. The author acknowledges support from the DFG through Grant 432680300 - SFB 1456 (C04).

\bibliography{lit}{}

\begin{thebibliography}{10}

\bibitem{holler22learning_parameter_id}
Christian Aarset, Martin Holler, and Tram Thi~Ngoc Nguyen.
\newblock {Learning-informed parameter identification in nonlinear
  time-dependent PDEs}.
\newblock {\em Applied Mathematics and Optimization}, 88:53 pp, 2023.

\bibitem{alougesKritsikisSteinerToussaint}
F.~Alouges, E.~Kritsikis, J.~Steiner, and J.-C. Toussaint.
\newblock A convergent and precise finite element scheme for
  {L}andau–{L}ifschitz–{G}ilbert equation.
\newblock {\em Numerische Mathematik}, 128(3):407--430, 2014.

\bibitem{BanasEtAl}
L.~Ba{\v n}as, M.~Page, D.~Praetorius, and J.~Rochat.
\newblock A decoupled and unconditionally convergent linear {FEM} integrator
  for the {L}andau-{L}ifshitz-{G}ilbert equation with magnetostriction.
\newblock {\em IMA Journal of Numerical Analysis}, 34:1361--1385, 2014.

\bibitem{Bard}
J.~Bard.
\newblock {\em {Practical Bilevel Optimization}}.
\newblock Kluwer Academic Publishers, Dordrecht, The Netherlands, 1998.

\bibitem{BartelsProhl1}
S.~Bartels and A.~Prohl.
\newblock Convergence of an implicit, constraint preserving finite element
  discretization of p-harmonic heat flow into spheres.
\newblock {\em Numerische Mathematik}, 109(4):489--507, 2008.

\bibitem{BurgerMuehlhuberIP}
M.~Burger and W.~M\"uhlhuber.
\newblock Iterative regularization of parameter identification problems by
  sequential quadratic programming methods.
\newblock {\em Inverse Problems}, 18:943--969, 2002.

\bibitem{BurgerMuehlhuberSINUM}
M.~Burger and W.~M\"uhlhuber.
\newblock Numerical approximation of an {SQP}-type method for parameter
  identification.
\newblock {\em SIAM J.\ Numer.\ Anal.}, 40:1775--1797, 2002.

\bibitem{cimrak}
Ivan Cimr\'ak.
\newblock {A Survey on the numerics and computations for the Landau-Lifshitz
  equation of micromagnetism}.
\newblock {\em Archives of Computational Methods in Engineering}, 15:277--309,
  2008.

\bibitem{Clason_2016}
Christian Clason, Barbara Kaltenbacher, and Daniel Wachsmuth.
\newblock Functional error estimators for the adaptive discretization of
  inverse problems.
\newblock {\em Inverse Problems}, 32(10):104004, aug 2016.

\bibitem{ClasonValkonen}
Christian Clason and Tuomo Valkonen.
\newblock Introduction to nonsmooth analysis and optimization.
\newblock {\em [arxiv:2001.00216v3]}, 2020.

\bibitem{Rieder:TCC21}
Matthias Eller and Andreas Rieder.
\newblock {Tangential cone condition and Lipschitz stability for the full
  waveform forward operator in the acoustic regime}.
\newblock {\em Inverse Problems}, 37(8), 2021.

\bibitem{Evans}
L.~C. Evans.
\newblock {\em Partial Differential Equations}.
\newblock Graduate Studies in Mathematics 19. AMS, Providence, RI, 1998.

\bibitem{Flemming}
J.~Flemming.
\newblock {Theory and examples of variational regularization with nonmetric
  fitting functionals}.
\newblock {\em J. Inverse Ill-Posed Probl.}, 18(6):677--699, 2010.

\bibitem{Frangos10}
M.~Frangos, Y.~Marzouk, K.~Willcox, and B.~van Bloemen~Waanders.
\newblock {\em {Surrogate and Reduced‐Order Modeling: A Comparison of
  Approaches for Large‐Scale Statistical Inverse Problems}}.
\newblock John Wiley \& Sons, Ltd, 2010.

\bibitem{tlg04}
T.~L. Gilbert.
\newblock A phenomenological theory of damping in ferromagnetic materials.
\newblock {\em IEEE Transactions on Magnetics}, 40(6):3443--3449, 2004.

\bibitem{HaAs01}
Eldad Haber and Uri~M Ascher.
\newblock Preconditioned all-at-once methods for large, sparse parameter
  estimation problems.
\newblock {\em Inverse Problems}, 17(6):1847, 2001.

\bibitem{Hanke97}
M.~Hanke.
\newblock {Regularizing properties of a truncated Newton-CG algorithm for
  nonlinear inverse problems}.
\newblock {\em Numerical Functional Analysis and Optimization},
  18(9-10):971--993, 1997.

\bibitem{HNS95}
M.~Hanke, A.~Neubauer, and O.~Scherzer.
\newblock A convergence analysis of the {L}andweber iteration for nonlinear
  ill-posed problems.
\newblock {\em Numer.\ Math.}, 72:21--37, 1995.

\bibitem{Harrach21}
Bastian Harrach.
\newblock {An Introduction to Finite Element Methods for Inverse Coefficient
  Problems in Elliptic PDEs}.
\newblock {\em Jahresbericht der Deutschen Mathematiker-Vereinigung},
  123:183--210, 2021.

\bibitem{Hinze_2019}
Michael Hinze and Tran Nhan~Tam Quyen.
\newblock {Finite element approximation of source term identification with
  TV-regularization}.
\newblock {\em Inverse Problems}, 35(12):124004, nov 2019.

\bibitem{HoffmanWaldNguyen:2021}
H.~Hoffmann, A.~Wald, and T.~T.~N. Nguyen.
\newblock Parameter identification for elliptic boundary value problems: an
  abstract framework and application.
\newblock {\em Inverse Problems}, 38(7):44 pp, 2022.

\bibitem{HubmerRamlau17}
Simon Hubmer and Ronny Ramlau.
\newblock {Convergence analysis of a two-point gradient method for nonlinear
  ill-posed problems}.
\newblock {\em Inverse Problems}, 33(9):095004, 2017.

\bibitem{HubmerScherzer:TCC18}
Simon Hubmer, Ekaterina Sherina, Andreas Neubauer, and Otmar Scherzer.
\newblock {Lamé Parameter Estimation from Static Displacement Field
  Measurements in the Framework of Nonlinear Inverse Problems}.
\newblock {\em SIAM Journal on Imaging Sciences}, 11(2), 2018.

\bibitem{Nakamura:MRE21}
Yu~Jiang, Gen Nakamura, and Kenji Shirota.
\newblock {Levenberg–Marquardt method for solving inverse problem of MRE
  based on the modified stationary Stokes system}.
\newblock {\em Inverse Problems}, 37(12), 2021.

\bibitem{JinZhou}
Bangti Jin and Zhi Zhou.
\newblock Error analysis of finite element approximations of diffusion
  coefficient identification for elliptic and parabolic problems.
\newblock {\em SIAM Journal on Numerical Analysis}, 59(1):119--142, 2021.

\bibitem{Kaltenbacher97}
B.~Kaltenbacher.
\newblock {Some Newton-type methods for the regularization of nonlinear
  ill-posed problems}.
\newblock {\em Inverse Problems}, 13(3):729--753, 1997.

\bibitem{aao16}
B.~Kaltenbacher.
\newblock Regularization based on all-at-once formulations for inverse
  problems.
\newblock {\em SIAM Journal of Numerical Analysis}, 54:2594--2618, 2016.

\bibitem{KKV14b}
B.~Kaltenbacher, A.~Kirchner, and B.~Vexler.
\newblock Goal oriented adaptivity in the {IRGNM} for parameter identification
  in {PDEs} {II}: all-at once formulations.
\newblock {\em Inverse Problems}, 30, 2014.
\newblock 045002.

\bibitem{KalNeuSch08}
B.~Kaltenbacher, A.~Neubauer, and O.~Scherzer.
\newblock {\em { Iterative Regularization Methods for Nonlinear Problems}}.
\newblock de Gruyter, Berlin, New York, 2008.
\newblock Radon Series on Computational and Applied Mathematics.

\bibitem{TCC21}
B.~Kaltenbacher, T.~T.~N. Nguyen, and O.~Scherzer.
\newblock {\em The tangential cone condition for some coefficient
  identification model problems in parabolic {PDEs}}.
\newblock Springer, 2021.

\bibitem{KNSW}
B.~Kaltenbacher, T.T.N. Nguyen, A.~Wald, and T.~Schuster.
\newblock {\em Parameter identification for the {L}andau-{L}ifshitz-{G}ilbert
  equation in Magnetic Particle Imaging}.
\newblock Springer, 2021.

\bibitem{Kaltenbacher:17}
Barbara Kaltenbacher.
\newblock All-at-once versus reduced iterative methods for time dependent
  inverse problems.
\newblock {\em Inverse Problems}, 33, 2017.

\bibitem{Kaltenbacher_2011}
Barbara Kaltenbacher, Alana Kirchner, and Boris Vexler.
\newblock {Adaptive discretizations for the choice of a Tikhonov regularization
  parameter in nonlinear inverse problems}.
\newblock {\em Inverse Problems}, 27(12):125008, nov 2011.

\bibitem{KarimiNutiniSchmidt}
Hamed Karimi, Julie Nutinia, and Mark Schmidt.
\newblock Linear convergence of gradient and proximal-gradient methods under
  the polyak-Łojasiewicz condition.
\newblock {\em ECML PKDD}, 9851:795--811, 2016.

\bibitem{Kindermann17}
Stefan Kindermann.
\newblock Convergence of the gradient method for ill-posed problems.
\newblock {\em Inverse Problems \& Imaging}, 11(4):703--720, 2017.

\bibitem{Kindermann21}
Stefan Kindermann.
\newblock On the tangential cone condition for electrical impedance tomography.
\newblock {\em Electronic Transaction on Numerical Analysis}, 57:17--34, 2022.

\bibitem{tktb12}
T.~Knopp and T.~M. Buzug.
\newblock {\em Magnetic Particle Imaging: an Introduction to Imaging Principles
  and Scanner Instrumentation}.
\newblock Springer Berlin Heidelberg, 2012.

\bibitem{llel92}
L.~Landau and E.~Lifshitz.
\newblock {3 - On the theory of the dispersion of magnetic permeability in
  ferromagnetic bodies Reprinted from Physikalische Zeitschrift der Sowjetunion
  8, Part 2, 153, 1935.}
\newblock In L.P. PITAEVSKI, editor, {\em {Perspectives in Theoretical
  Physics}}, pages 51--65. Pergamon, Amsterdam, 1992.

\bibitem{Neubauer17}
Andreas Neubauer.
\newblock {On Nesterov acceleration for Landweber iteration of linear ill-posed
  problems}.
\newblock {\em Journal of Inverse and Ill-posed Problems}, 25(3):381--390,
  2017.

\bibitem{NeubauerScherzer90}
Andreas Neubauer and Otmar Scherzer.
\newblock {Finite-dimensional approximation of tikhonov regularized solutions
  of non-linear ill-posed problems}.
\newblock {\em Numerical Functional Analysis and Optimization},
  11(1--2):85--99, 1990.

\bibitem{Nguyen:19}
T.~T.~N. Nguyen.
\newblock {Landweber{\textendash}Kaczmarz for parameter identification in
  time-dependent inverse problems: all-at-once versus reduced version}.
\newblock {\em Inverse Problems}, 35(3):035009, 2019.

\bibitem{NguyenWald:2022}
T.~T.~N. Nguyen and A.~Wald.
\newblock {On numerical aspects of parameter identification for the
  Landau-Lifshitz-Gilbert equation in Magnetic Particle Imaging}.
\newblock {\em Inverse Problems and Imaging}, 16(1), 2022.
\newblock Doi:10.3934/ipi.2021042.

\bibitem{Rieder99}
A.~Rieder.
\newblock {On the regularization of nonlinear ill-posed problems via inexact
  Newton iterations}.
\newblock {\em Inverse Problems}, 15(1):309--327, 1999.

\bibitem{Roubicek}
T.~{Roub\' i\v cek}.
\newblock {\em Nonlinear Partial Differential Equations with Applications}.
\newblock Springer Basel, 2013.

\bibitem{scherzer95}
O.~Scherzer.
\newblock {On Convergence Criteria of Iterative Methods Based on Landweber
  Iteration for Solving Nonlinear Problems}.
\newblock {\em Journal of Mathematical Analysis and Applications}, 194(3),
  1995.

\bibitem{ScherzerHofmannNashed}
Otmar Scherzer, Bernd Hofmann, and Zuhair Nashed.
\newblock {Newton’s methods for solving linear inverse problems with neural
  network coders}.
\newblock {\em arXiv:2303.14058v1[math.FA]}.

\bibitem{Troeltzsch}
F.~Tr\"oltzsch.
\newblock {\em Optimal Control of Partial Differential Equations: Theory,
  Methods and Applications}.
\newblock American Mathematical Society, 2010.

\bibitem{LeHe16}
T~van Leeuwen and F~J Herrmann.
\newblock A penalty method for {PDE}-constrained optimization in inverse
  problems.
\newblock {\em Inverse Problems}, 32(1):015007, 2016.

\end{thebibliography}
\bibliographystyle{plain} 
\end{document}